\documentclass[a4paper, UKenglish, cleveref, autoref, thm-restate]{lipics-v2021}

\graphicspath{{./figs/}}

\usepackage[normalem]{ulem}
\usepackage{tikz}
\usepackage{mathrsfs}
\usepackage{enumerate}
\usepackage{mathdots}
\usepackage{todonotes}
\usepackage{algorithm}  
\usepackage{algpseudocode} 

\newcommand{\cvec}{\mathrm{\mathbf{vec}}}
\newcommand{\Pvec}{\mathrm{\mathbf{Pvec}}}
\newcommand{\SPvec}{\mathrm{\mathbf{SPvec}}}
\newcommand{\barc}{\mathrm{barc}}
\newcommand{\kf}{\mathrm{\mathbf{k}}}

\newcommand{\mrk}{\mathrm{\mathbf{mrk}}}
\newcommand{\mdgm}{\mathrm{\mathbf{mdgm}}}
\newcommand{\rank}{\mathrm{rank}}

\newcommand{\im}{\mathrm{im}}
\newcommand{\Int}{\mathrm{\mathbf{Int}}}
\newcommand{\id}{\mathrm{id}}

\newcommand{\Z}{\mathbb{Z}}
\newcommand{\R}{\mathbb{R}}

\newcommand{\beps}{\overline{\epsilon}}

\newcommand{\dI}{\mathrm{d_I}}

\newcommand{\dE}{\mathrm{d_E}}
\newcommand{\dES}{\mathrm{d}_\mathrm{E}^S}

\newcommand{\PN}{\mathrm{PN}}
\newcommand{\Dgm}{\mathrm{Dgm}}

\newcommand{\para}[1]{\noindent{\textbf{\sffamily{#1}}}}

\title{Meta-Diagrams for 2-Parameter Persistence} 

\titlerunning{Meta-Diagrams for Persistence}

\author{Nate Clause}{Ohio State University, Columbus, OH, USA}{clause.15@osu.edu}{}{NC is partially supported by NSF CCF 1839356 and 
NSF DMS 1547357.}

\author{Tamal K. Dey}{Purdue University, West Lafayette, IN, USA}{tamaldey@purdue.edu}{}{TD is partially supported by NSF CCF 2049010.}

\author{Facundo M\'{e}moli}{Ohio State University, Columbus, OH, USA}{memoli@math.osu.edu}{}{FM is partially supported by BSF 2020124, NSF CCF 1740761, NSF CCF 1839358, and NSF IIS 1901360.}

\author{Bei Wang}{University of Utah, Salt Lake City, UT, USA}{beiwang@sci.utah.edu}{}{BW is partially supported by NSF IIS 2145499, NSF IIS 1910733, and DOE DE SC0021015.}

\authorrunning{N.\,Clause, T.\,K.\,Dey, F.\,M\'{e}moli and B.\,Wang} 

\Copyright{Nate Clause, Tamal K. Dey, Facundo M\'{e}moli, and Bei Wang} 



\category{} 

\relatedversion{} 

\acknowledgements{}
\nolinenumbers 

\EventEditors{Erin W. Chambers and Joachim Gudmundsson}
\EventNoEds{2}
\EventLongTitle{39th International Symposium on Computational Geometry (SoCG 2023)}
\EventShortTitle{SoCG 2023}
\EventAcronym{SoCG}
\EventYear{2023}
\EventDate{June 12--15, 2023}
\EventLocation{Dallas, Texas, USA}
\EventLogo{}
\SeriesVolume{258}
\ArticleNo{27}

\begin{document}

\maketitle

\begin{abstract}
We first introduce the notion of meta-rank for a 2-parameter persistence module, an invariant that captures the information behind images of morphisms between 1D slices of the module. 
We then define the meta-diagram of a 2-parameter persistence module to be the M\"{o}bius inversion of the meta-rank, resulting in a function that takes values from signed 1-parameter persistence modules. 
We show that the meta-rank and meta-diagram contain information equivalent to the rank invariant and the signed barcode. This equivalence leads to computational benefits, as we introduce an algorithm for computing the meta-rank and meta-diagram of a 2-parameter module $M$ indexed by a bifiltration of $n$ simplices in $O(n^3)$ time. This implies an improvement upon the existing algorithm for computing the signed barcode, which has $O(n^4)$ runtime. This also allows us to improve the existing upper bound on the number of rectangles in the rank decomposition of $M$ from $O(n^4)$ to $O(n^3)$. 
In addition, we define notions of erosion distance between meta-ranks and between meta-diagrams, and show that under these distances, meta-ranks and meta-diagrams are stable with respect to the interleaving distance. 
Lastly, the meta-diagram can be visualized in an intuitive fashion as a persistence diagram of diagrams, which generalizes the well-understood persistence diagram in the 1-parameter setting.

\end{abstract}

\section{Introduction}
\label{sec:introduction} 

In the case of a 1-parameter persistence module, the persistence diagram (or barcode) captures its complete information up to isomorphism via a collection of intervals. 
The persistence diagram is represented as a multi-set of points in the plane, whose coordinates are the birth and death times of intervals, each of which encodes the lifetime of a topological feature. 
This compact representation of a persistence module enables its  interpretability and facilitates its visualization.
When moving to the multiparameter setting, the situation becomes much more complex as a multiparameter persistence module may contain indecomposable pieces that are not entirely determined by  intervals or do not admit a finite discrete description~\cite{carlsson2009theory}.

Such an increased complexity has led to the study of other invariants for multiparameter persistence modules.
The first invariant is the \emph{rank invariant}~\cite{carlsson2009theory}, which captures the information from the  images of internal linear maps in a persistence module across all  dimensions. 
Patel noticed that the persistence diagram in the 1-parameter setting is equivalent to the \emph{M\"{o}bius inversion}~\cite{rota1964foundations} of the rank function~\cite{patel2018generalized}.
He then defined the generalized persistence diagram as the M\"{o}bius inversion of a function defined on a subset of intervals of $\R$, denoted $\Dgm$, with values in some abelian group.  

The idea of M\"{o}bius inversion has been extended in many directions. 
Kim and M\'{e}moli defined generalized persistence diagrams for modules on posets \cite{clause2022discriminating,kim2021generalized}.
Patel and McCleary extended Patel's generalized persistence diagrams to work for persistence modules indexed over finite lattices \cite{mccleary2022edit}.
Botnan et al. \cite{botnan2022rectangle} implicitly studied the M\"{o}bius inversion of the rank function for 2-parameter modules, leading to a notion of a diagram with domain all rectangles in $\Z^2$.
Asashiba et al. used M\"{o}bius inversion on a finite 2D grid to define interval-decomposable approximations \cite{asashiba2019approximation}.
Morozov and Patel \cite{morozov2021output} defined a generalized persistence diagram in the 2-parameter setting via M\"{o}bius inversion of the birth-death function and provided an algorithm for its computation.
Their algorithm has some similarity with ours: it utilizes the vineyards algorithm \cite{cohen2006vines} to study a 2-parameter persistence module, by slicing it over 1D paths.

Our work also involves the idea of slicing a 2-parameter module.
This idea of slicing appears in the fibered barcode \cite{cerri2013betti, lesnick2015interactive}, which is equivalent to the rank function.
To obtain insight into the structure of a 2-parameter persistence module $M$, Lesnick and Wright~\cite{lesnick2015interactive} explored a set of 1-parameter modules obtained via restricting $M$ onto all possible lines of non-negative slope. 
Buchet and Escolar~\cite{BuchetEscolar2020} showed that any 1-parameter persistence module with finite support could be found as a restriction of some indecomposable 2-parameter persistence module with finite support. 
Furthermore, Dey et al. \cite{DKM22} showed that certain zigzag (sub)modules
of a 2-parameter module can be used to compute the generalized rank invariant, whose M\"{o}bius inversion is the generalized persistence diagram defined by Kim and M\'{e}moli. 
Our work considers the images between slices of a 2-parameter module, which is related to the work by Bauer and Schmal \cite{bauer2022efficient}.    

In \cite{botnan2022signed}, Botnan et al.~introduced the notion of \emph{rank decomposition}, which is equivalent to the generalized persistence diagram formed by M\"{o}bius inversion of the rank function, under some additional conditions. 
Botnan et al. further demonstrated that the process of converting a module to a rank decomposition is stable with respect to the matching distance~\cite{landi2018rank}.
Additionally, they introduced a visualization of this rank decomposition via a \emph{signed barcode}, which highlights the diagonals of rectangles appearing in the rank decomposition, along with their multiplicity.
They visualized the value of the signed barcode with a $2$-parameter persistence module generated by clustering a point cloud with a scale and a density parameter.

\begin{figure}[t]
    \centering
    \includegraphics[width=.45\linewidth]{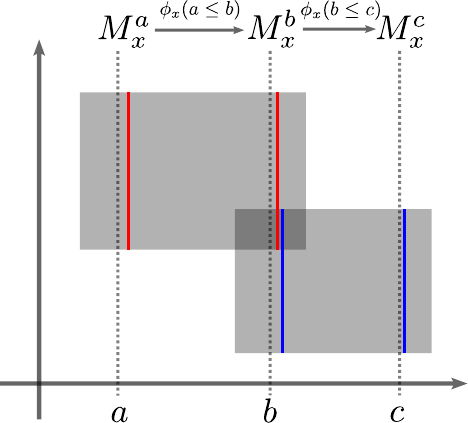}
    \caption{Slicing a 2-parameter module $M$ along vertical lines yields 1-parameter modules, such as $M_x^a, M_x^b$, and $M_x^c$. 
    There are morphisms between these 1-parameter modules induced by the internal morphisms of $M$, and the meta-rank captures the information about these morphisms.
    For example, if $M$ is defined as the direct sum of the two interval modules given by the two shaded rectangles, then the meta-rank of $M$ on $[a,b)$ is the image of $\phi_x(a\leq b)$, which has a barcode consisting of the red interval.
   The meta-rank of $M$ on $[b,c)$ has a barcode consisting of the blue interval, and the meta-rank of $M$ on $[a,c)$ is $0$, as $\phi_x(a\leq c) = \phi_x(b\leq c)\circ \phi_x(a\leq b)=0$.}
    \label{fig:mrk-example-1}
\end{figure}

Unlike the previous results that perform M\"{o}bius inversion over a higher-dimensional poset such as $\Z^2$, our work involves M\"{o}bius inversion over a finite subcollection of intervals of $\R$, which leads to a simpler inversion formula.
In this work, we introduce the notion of \emph{meta-rank} for a $2$-parameter persistence module, which is a map from 
$\Dgm$ to isomorphism classes of persistence modules.
Instead of looking at images of linear maps between vector spaces (as with the usual rank invariant), the meta-rank considers images of the maps between 1-parameter persistence modules formed by slicing a $2$-parameter persistence module along vertical and horizontal lines, see~\autoref{fig:mrk-example-1}.
We then define the meta-diagram as the M\"{o}bius inversion of the meta-rank, giving a map from $\Dgm$ to isomorphism classes of signed persistence modules.
This contrasts Botnan et al.'s approach~\cite{botnan2022signed} of using M\"{o}bius inversion in 2D, as our M\"{o}bius inversion formula over $\Dgm$ is simpler and consists of fewer terms. 
	
\begin{figure}[t]
    \centering
    \includegraphics[width=0.6\linewidth]{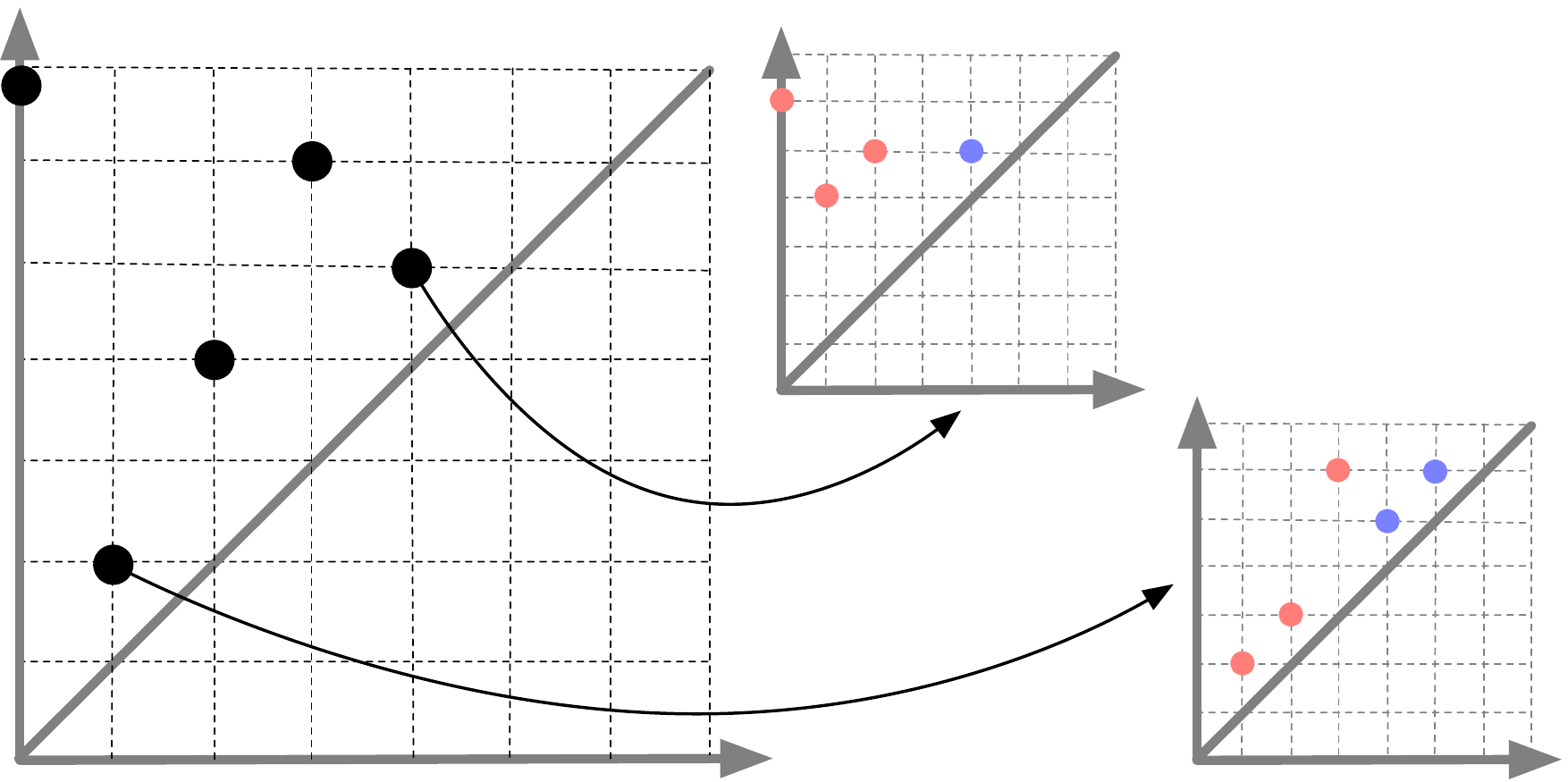}
    \caption{A meta-diagram viewed as a persistence diagram of signed diagrams (red and blue mean positive and negative signs respectively).}
    \label{fig:diagram-of-diagrams}
\end{figure}

\para{Contributions.}
The meta-rank and meta-diagram turn out to contain information equivalent to the rank invariant (\autoref{prop:mrk-rk-equivalence}) and signed barcode (\autoref{prop:mdgm-rkdecompequiv}) respectively. 
Therefore, both meta-rank and meta-diagram can be regarded as these known invariants seen from a different perspective.
However, this different viewpoint brings forth several advantages as listed below that
make the meta-rank and meta-diagram stand out on their own right:
\begin{enumerate}
\item The meta-rank and meta-diagram of a $2$-parameter persistence module $M$ induced by a bifiltration of a simplicial complex with $n$ simplices can be computed in $O(n^3)$ time. 
\item This immediately implies an improvement of the $O(n^4)$ algorithm of Botnan et al.~\cite{botnan2022signed} for computing the signed barcodes.
\item The $O(n^3)$ time algorithm for computing meta-rank and meta-diagram also implicitly improves
    the bound on the number of signed bars in the rank decomposition of $M$ to $O(n^3)$ from the current known bound of $O(n^4)$. This addresses an open question whether
    the size of the signed barcode is bounded tightly
    by the number of rectangles or not.
\item The meta-diagram can be viewed as a persistence diagram of signed diagrams as illustrated in~\autoref{fig:diagram-of-diagrams}. Such an intuitive visualization generalizes the classic persistence diagram -- a known technique in topological data analysis (TDA) -- to summarize persistent homology.  
\item The meta-diagram also generalizes the concept of a sliced barcode well-known in TDA~\cite{lesnick2015interactive}. It  ensembles sliced bars on a set of lines, but not forgetting the maps between slices induced by the module $M$ being sliced.
\end{enumerate}

\section{Preliminaries}
\label{sec:prelim}

We regard a poset $(P,\leq)$ as a category, with objects the elements $p\in P$, and a unique morphism $p\to q$ if and only if $p\leq q$; this is referred to as the \emph{poset category} for $(P,\leq)$. When it is clear from the context, we will denote the poset category by $P$.

Fix a field $\kf$, and assume all vector spaces have coefficients in $\kf$ throughout this paper.
Let $\cvec$ denote the category of finite-dimensional vector spaces with linear maps between them. 
A \emph{persistence module}, or \emph{module} for short, is a functor $M:P\to \cvec$.
For any $p\in P$, we denote the vector space $M_p:= M(p)$, and for any $p\leq q\in P$, we denote the linear map $\varphi_M(p\leq q):= M(p\leq q)$.
When $M$ is apparent, we drop the subscript from $\varphi_M$.
We call $P$ the \emph{indexing poset} for $M$. 
We focus on the cases when the indexing poset is $\R$ or $\R^2$, equipped with the usual order and product order, respectively.
Definitions and statements we make follow analogously when the indexing poset is $\Z$ or $\Z^2$, which we will cover briefly in~\autoref{sec:algorithms}. 
If the indexing poset for $M$ is $P\subseteq\R$, then $M$ is a \emph{1-parameter (or 1D) persistence module}.
If the indexing poset for $M$ is $P\subseteq\R^2$, with $P$ not totally-ordered, then $M$ is a \emph{2-parameter (or 2D) persistence module}, or a \emph{bimodule} for short.

Following~\cite{mccleary2020bottleneck}, we  require that persistence modules be \emph{constructible}: 
\begin{definition}
\label{def:constructible}
A module $M:\R\to \cvec$ is \emph{constructible} if there exists a finite set $S:=\{s_1<\ldots<s_n\}\subset \R$ such that:
\begin{itemize}
    \item For $a<s_1$, $M(a)=0$;
    \item For $s_i\leq a\leq b<s_{i+1}$, $\varphi_M(a\leq b)$ is an isomorphism;
    \item For $s_n\leq a\leq b$, $\varphi_M(a\leq b)$ is an isomorphism.
\end{itemize}
Similarly, a bimodule $M:\R^2\to \cvec$ is \emph{constructible} if there exists a finite set
$S:=\{s_1<\ldots<s_n\}\subset \R$ such that:
\begin{itemize}
    \item If $x<s_1$ or $y<s_1$, then $M((x,y))=0$,
    \item For $s_i\leq x_1\leq x_2<s_{i+1}$ and $s_j\leq y_1\leq y_2<s_{j+1}$, $\varphi_M((x_1,y_1)\leq (x_2,y_2))$ is an isomorphism,
    \item If $x_1\geq s_n$ or $y_1\geq s_n$ and $(x_1,y_1)\leq (x_2,y_2)$, then $\varphi_M((x_1,y_1)\leq (x_2,y_2))$ is an isomorphism.
\end{itemize}
In either case, such a module is \emph{$S$-constructible}.
\end{definition}

If a module is $S$-constructible, unless otherwise stated, assume $S=\{s_1<\ldots<s_n\}$. 
If $M$ is $S$-constructible, then $M$ is $S'$-constructible for any $S'\supseteq S$.
For the rest of the paper, we assume any given persistence module is constructible.

Of particular importance in the study of 1- and 2-parameter  persistence modules are the notions of interval modules and interval decomposable modules.
We state the definitions:

\begin{definition}
For a poset $(P,\leq)$, an \emph{interval} of $P$ is a non-empty subset $I\subset P$ such that:
\begin{itemize}
    \item (convexity) If $p,r\in I$ and $q\in P$ with $p\leq q\leq r$, then $q\in I$.
    \item (connectivity) For any $p,q\in I$, there is a sequence $p=r_0,r_1,\ldots,r_n=q$ of elements of $I$, where for all $0\leq i\leq n-1$, either $r_i\geq r_{i+1}$ or $r_i\leq r_{i+1}$. 
\end{itemize}
We denote the collection of all intervals of $P$ as $\Int(P)$.
\end{definition}

For $I\in \Int(P)$, the \emph{interval module} $\kf^I$ is the persistence module indexed over $P$, with:
\[\kf^I_p = \begin{cases} \kf & \mathrm{if \, } p\in I\\
0 & \mathrm{otherwise}
\end{cases},
\hspace{10mm}
\varphi_{\kf^I}(p\leq q) = \begin{cases} \id_\kf & \mathrm{if \, } p\leq q\in I\\
0 & \mathrm{otherwise}\end{cases}.\]

Given any $M,N:P\to \cvec$, the \text{direct sum} $M\oplus N$ is defined point-wise at each $p\in P$.
We say a nontrivial $M:P\to \cvec$ is \emph{decomposable} if $M$ is isomorphic to $N_1\oplus N_2$ for some non-trivial $N_1,N_2:P\to \cvec$, which we denote by $M\cong N_1\oplus N_2$.
Otherwise, $M$ is \emph{indecomposable}.
Interval modules are indecomposable \cite{botnan2018algebraic}.

A persistence module $M:P\to \cvec$ is \emph{interval decomposable} if it is isomorphic to a direct sum of interval modules. 
That is, if there is a multiset of intervals $\barc(M)$, such that:
\[M \cong \bigoplus_{I\in \barc(M)} \kf^I\]
If this multiset exists, we call it the \emph{barcode} of $M$. 
If it exists, $\barc(M)$ is well-defined as a result of the Azumaya-Krull-Remak-Schmidt theorem \cite{azumaya1950corrections}. 
Thus, in the case where $M$ is interval decomposable, $\barc(M)$ is a complete descriptor of the isomorphism type of $M$.

Of particular importance in this work are \emph{right-open rectangles}, which are intervals $R\subset \R^2$ of the form $R=[a_1,b_1)\times [a_2,b_2)$.
If $M$ can be decomposed as a direct sum of interval modules $\kf^R$ with $R$ a right-open rectangle, then we say $M$ is \emph{rectangle decomposable}.

1-parameter persistence modules are particularly nice, as they are always interval decomposable \cite{crawley2015decomposition}. 
As a result, the barcode is a complete invariant for 1-parameter  persistence modules.
On the other hand, bimodules do not necessarily decompose in this way.
In fact, there is no complete discrete descriptor for bimodules~\cite{carlsson2009theory}.

A number of invariants have been proposed to study bimodules. 
One of the first and the most notable invariant is the \emph{rank invariant} \cite{carlsson2009theory} recalled in~\autoref{def:rank-invariant}. 
\begin{definition}[\cite{carlsson2009theory}]
\label{def:rank-invariant}
For $P$ a poset, define $\mathbf{D}(P):=\{(a,b)\in P\times P  \, | \, a\leq b\}$. For $M:P\to \cvec$, the \emph{rank invariant of} $M$, $\rank_M:\mathbf{D}(P)\to \Z_{\geq 0}$, is defined point-wisely as: 
\[\rank_M(a,b):= \rank(\varphi_M(a\leq b))\]
\end{definition}

For a bimodule, the rank invariant is inherently a 4D object, making it difficult to visualize directly. 
RIVET \cite{lesnick2015interactive} visualizes the rank invariant indirectly through the fibered barcode. 
In \cite{botnan2022signed}, Botnan et al. defined the \emph{signed barcode} based on the notion of a \emph{rank decomposition}: 
\begin{definition}[\cite{botnan2022signed}]
\label{def:rankdecomp}
Let $M:\R^n\to \cvec$ be a persistence module with rank function $\rank_M$.
Suppose $\mathscr{R},\mathscr{S}$ are multisets of intervals from $\R^n$.
Define $\kf_\mathscr{R}:= \oplus_{I\in \mathscr{R}} \kf^R$, and similarly $\kf_\mathscr{S}$.
Then $(\mathscr{R},\mathscr{S})$ is a \emph{rank decomposition} for $\rank_M$ if as integral functions:
\[\rank_M = \rank_\mathscr{R}-\rank_\mathscr{S}.\]
\end{definition}

If $\mathscr{R},\mathscr{S}$ consist of right-open rectangles, then  the pair is a rank decomposition by rectangles.
We have:
\begin{theorem}[\cite{botnan2022signed}, Theorem 3.3]
Every finitely presented $M:\R^2\to \cvec$ admits a unique minimal rank decomposition by rectangles.
\end{theorem}
Here minimality comes in the sense that $\mathscr{R}\cap \mathscr{S}=\emptyset$.
The signed barcode then visualizes the rank function in $
\R^2$ by showing the diagonals of the rectangles in $\mathscr{R}$ and $\mathscr{S}$.

\section{Meta-Rank}
\label{sec:meta-rank}

In this section, we introduce the \emph{meta-rank}.
While the rank invariant captures the information of images between pairs of vector spaces in a persistence module, the meta-rank captures the information of images between two 1-parameter persistence modules obtained via slicing a bimodule. 
We describe the results for modules over $\R^2$ and $\R$, but they hold in direct analogue in the $\Z^2$ and $\Z$ setting, which are  briefly covered in \autoref{sec:algorithms}.
For omitted proofs, see \autoref{sec:meta-rank-details}.
We begin with some preliminary definitions:

\begin{definition}
Let $M:\R^2\to \cvec$ be a bimodule.
For $s\in \R$, define the vertical slice $M^s_x:\R\to \cvec$ point-wise as $M^s_x(a) := M(s,a)$, and with morphisms from $a$ to $b$ as $\varphi^s_x(a\leq b):= \varphi((s,a)\leq (s,b))$.
Analogously, define the horizontal slice $M^s_y:\R\to \cvec$ by setting $M^s_y(a):= M(a,s)$ and $\varphi^s_y(a\leq b):= \varphi((a,s)\leq (b,s))$ for all $a\leq b\in \R$.
\end{definition}

Define a morphism of 1-parameter persistence modules $\phi_x(s\leq t):M^s_x\to M^t_x$ for $s\leq t\in \R$ by $\phi_x(s\leq t)(a):= \varphi((s,a)\leq(t,a))$. 
Analogously, define $\phi_y(s\leq t):M^s_y\to M^t_y$ for $s\leq t\in \R$ by $\phi_y(s\leq t)(a):= \varphi((a,s)\leq(a,t))$.

Denote by $\Pvec$ the isomorphism classes of persistence modules over $\R$. 
Each element of $\Pvec$ can be uniquely represented by its barcode, which is what we do in practice.
We recall the definition of $\Dgm$ from \cite{patel2018generalized}, which serves as the domain for the meta-rank:
\begin{definition}[\cite{patel2018generalized}]\label{def:dgm}
Define $\Dgm$ as the poset of all half-open intervals $[p,q)\subset \R$ for $p<q$, and all half-infinite intervals $[p,\infty)\subset \R$. 
The poset relation is inclusion.
\end{definition}

\begin{definition}
\label{def:mrk}
Suppose $M:\R^2\to \cvec$ is $S$-constructible.
Define the \emph{horizontal meta-rank} $\mrk_{M,x}:\Dgm\to \Pvec$ as follows:
\begin{itemize}
    \item For $I=[s,s_i)$ with $s_i\in S$, $\mrk_{M,x}(I):= [\im(\phi_x(s\leq s_i-\delta))]$, for some $\delta>0$ such that $s_i-\delta\geq s$ and $s_i-\delta\geq s_{i-1}$.
    \item For $I=[s,\infty)$, $\mrk_{M,x}(I):=[\im(\phi_x(s\leq s_n))]$.
    \item For all other $I=[s,t)$, $\mrk_{M,x}(I):=[\im(\phi_x(s\leq t))]$.
\end{itemize} 
Analogously, define the \emph{vertical meta-rank}, $\mrk_{M,y}:\Dgm\to \Pvec$ by replacing each instance of $x$ above with $y$.
\end{definition}

The results in this paper are stated in terms of the horizontal meta-rank, but hold analogously for the vertical meta-rank.
To simplify notation, we henceforth denote $\mrk_{M,x}$ as $\mrk_M$.
When there is no confusion, we drop the subscript from $\mrk_M$. 

\begin{figure}[t]
    \centering
    \includegraphics[width=0.3\linewidth]{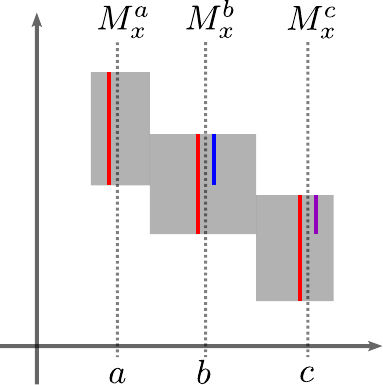}
    \caption{An illustration of $M$ and its barcode for some values of $\mrk_M$ in \autoref{example:mrk-def-ex}.}
    \label{fig:mrk-example-2}
\end{figure}

\begin{example}
\label{example:mrk-def-ex}
As illustrated in \autoref{fig:mrk-example-2}, let $I$ be the single gray interval and define the bimodule $M:=\kf^I$.
The barcodes for the 1-parameter modules $M_x^a,M_x^b$, and $M_x^c$ are shown in red next to their corresponding vertical slices.
The barcode for $\mrk_M([a,b))$ consists of the blue interval, which is the overlap of the bars in $M_x^a$ and $M_x^b$, 
$\barc(M_x^a) \cap \barc(M_x^b)$.
Similarly, $\mrk_M([b,c))$ has a barcode consisting of the purple interval, which is the overlap of the bars in $M_x^b$ and $M_x^c$.
As the bars in the barcodes for $M_x^a$ and $M_x^c$ have no overlap, $\im(\phi_x(a\leq c))=0$, therefore $\mrk_M([a,c))=0$.
\end{example}

\begin{remark}
\label{rmk:mrkxynotequal}
In general, $\mrk_x\neq \mrk_y$. 
For example, consider the right-open rectangle $R$ with the lower-left corner the origin, and the upper right corner $(1,2)$, as in \autoref{fig:mrkx-mrky-ex}.
Let $M:=\kf^R$.
As illustrated, $\mrk_{M,x}([0,1)) = [0,2)\neq [0,1)=\mrk_{M,y}([0,1))$.
\end{remark}

\begin{figure}[t]
    \centering
    \includegraphics[width=.6
    \linewidth]{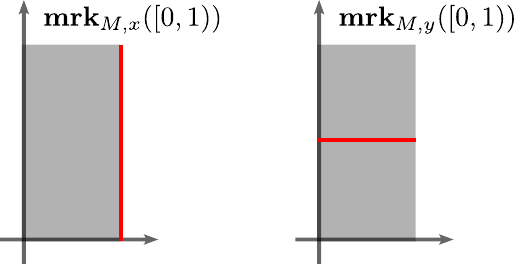}
    \caption{An illustration of $M$, depicting $\mrk_{M,x}([0,1))\neq \mrk_{M,y}([0,1))$.} 
    \label{fig:mrkx-mrky-ex}
\end{figure}

The following~\autoref{prop:direct-sums} allows us to compute the meta-rank of a bimodule via the meta-ranks of its indecomposable summands:
\begin{proposition}
\label{prop:direct-sums}
For $M,N:\R^2\to \cvec$, we have:
\[\mrk_M\oplus \mrk_N = \mrk_{M\oplus N}\]
\end{proposition}
where $\mrk_M\oplus \mrk_N:\Dgm\to \Pvec$ is defined as:
\[(\mrk_M\oplus \mrk_N)([s,t)) := [\mrk_M([s,t))\oplus \mrk_N([s,t))].\]

For a finite $S\subseteq \R$, let $\overline{S}:=S\cup\{\infty\}$.
Define $\overline{S}_>:\R\cup \{\infty\}\to \overline{S}$ as $\overline{S}_>(t):=\min \{s\in \overline{S} \, | \, s>t\}$.
For $M\in \Pvec$, $[b,d)\in \Dgm$, let $\#[b,d)\in M$ denote the multiplicity of $[b,d)\in \barc(M)$. 
The rank invariant and the meta-rank contain equivalent information:
\begin{proposition}
\label{prop:mrk-rk-equivalence}
For $M:\R^2\to \cvec$, one can compute $\rank_M$ from $\mrk_M$ and one can compute $\mrk_M$ from $\rank_M$.
In particular, given $(s,y)\leq (t,y')\in \R^2$, 
\[\rank_M((s,y),(t,y')) = \#[b_i,d_i)\in \mrk_M([s,\overline{S}_>(t))) \,  \, s.t. \, \, b_i\leq y\leq y'<d_i.\]
That is, the rank is the number of intervals in $\barc(\mrk_M([s,\overline{S}_>(t))))$ containing $[y,y']$.  
\end{proposition}
The reason for needing $\overline{S}_>(t)$ for the right endpoint is that if $t\in S$, $\mrk_M([s,t))$ does not capture the information of the image of $\phi_x(s\leq t)$, only the image of $\phi_x(s\leq t-\delta)$.

Finally, we discuss the stability of the meta-rank. 
The meta-rank is stable with respect to a notion of erosion distance, based on that of Patel \cite{patel2018generalized}.
We introduce truncated barcode:

\begin{definition}
\label{def:truncated-barcode}
For $\epsilon\geq 0$, and $I=[s,t)\in \Dgm$, define $I[\epsilon:]:= [s+\epsilon,t)$. 
For $M:\R\to \cvec$ define:
$ \barc_\epsilon(M) := \{I[\epsilon:] \, | \, I\in \barc(M)\}$. 
If $I=[s,t) \in \barc(M)$ has $t-s\leq\epsilon$, then $I$ has no corresponding interval in $\barc_\epsilon(M)$.
\end{definition}

\begin{definition}
\label{def:compare}
For $M,N:\R\to \cvec$, we say $M\preceq_\epsilon N$ if there exists an injective function on barcodes $\iota:\barc_\epsilon(M)\hookrightarrow \barc(N)$ such that for all $J\in \barc_\epsilon(M)$, $J\subseteq \iota(J)$.
\end{definition}

For $\epsilon\geq 0$, $M\in \Pvec$, let $M^\epsilon$ refer to the $\epsilon$-shift of $M$ \cite{lesnick2015theory}, with $M^\epsilon(a):=M(a+\epsilon)$ and $\varphi_{M^\epsilon}(a\leq b):=\varphi_M(a+\epsilon\leq b+\epsilon)$.
For $I=[s,t)\in \Dgm$ and $a,b\in \R$, let $I^b_a:=[s+a,t+b)$, with the convention $\infty+b:=\infty$ for any $b\in \R$.
We now define the erosion distance:

\begin{definition}
\label{def:erosion-distance}
Let $M,N:\R^2\to \cvec$. Define the erosion distance as follows:
\begin{align*}
\dE(\mrk_M,\mrk_N):=\inf \{\epsilon>0 \, | \, \forall I\in \Dgm, \, &\mrk_M(I^\epsilon_{-\epsilon})^\epsilon \preceq_{2\epsilon} \mrk_N(I) \, \mathrm{and} \\
&\mrk_N(I^\epsilon_{-\epsilon})^\epsilon \preceq_{2\epsilon} \mrk_M(I)\}
\end{align*}
if the set we are infimizing over is empty, we set $\dE(\mrk_M,\mrk_N):= \infty$.
\end{definition}
\begin{proposition}
\label{prop:mrk-pseudometric}
$\dE$ as defined in \autoref{def:erosion-distance} is an extended pseudometric on the collection of meta-ranks of constructible bimodules $M:\R^2\to \cvec$.
\end{proposition}

We compare bimodules $M$ and $N$ using the  multiparameter interleaving distance \cite{lesnick2015theory}.
The $\epsilon$-shift and the truncation of the barcode in  \autoref{def:truncated-barcode} are necessary for stability, due to the interleaving distance being based on diagonal shifts of bimodules, whereas the meta-rank is based on horizontal maps instead of diagonal ones.
We have the following:
\begin{theorem}
\label{thm:mrk-stability}
For constructible $M,N:\R^2\to \cvec$, we have:
\[\dE(\mrk_M, \mrk_N) \leq \dI(M,N)\]
\end{theorem}

\section{Meta-Diagram}
\label{sec:meta-diagram}

We use the M\"{o}bius inversion formula from Patel \cite{patel2018generalized} on the meta-rank function to get a \emph{meta-diagram}. 
This formula involves negative signs, so we need a notion of signed persistence modules. 
Our ideas are inspired by the work of Betthauser et al. \cite{betthauser2022graded}, where we consider breaking a function into positive and negative parts.
For omitted proofs, see \autoref{sec:meta-diagram-details}. 

 \begin{definition}
 A \emph{signed 1-parameter persistence module} is an ordered pair $(M,N)$, where $M,N:\Z\to \cvec$ are 1-parameter persistence modules. 
$M$ is the \emph{positively signed} module, and $N$ is the \emph{negatively signed} module.
 \end{definition}

\begin{definition}
View $\Pvec$ as a commutative monoid with operation $\oplus$ given by $[M]\oplus[N]:= [M\oplus N]$, and identity element $[0]$.
Define $\SPvec$ to be the Grothendieck group of $\Pvec$.
\end{definition}

Each element of $\SPvec$ is an isomorphism class of ordered pairs $[([M^+],[M^-])]$.
From the completeness of barcodes for 1-parameter persistence modules, we assume without loss of generality that each element $M^+$, $M^-$ is given by $\ast:=\oplus_{I\in \barc(\ast)}\kf^I$ and drop the internal equivalence class notation to write an element of $\SPvec$ as $[(M^+,M^-)]$.
\autoref{prop:mdgmrepresentatives} allows us to make a canonical choice of representative for each element of $\SPvec$: 
\begin{proposition}
\label{prop:mdgmrepresentatives}
Let $A\in \SPvec$. Then there is a unique representative $A=[(M^+,M^-)]$ with $\barc(M^+)\cap \barc(M^-)=\emptyset$.
\end{proposition}

As a result of \autoref{prop:mdgmrepresentatives}, when convenient,  we represent an element of $\SPvec$ uniquely by the sum of barcodes of this special representative, as in the following example:

\begin{example}
Consider $[(N^+,N^-)]\in \SPvec$ where $\barc(N^+) = \{[0,4],[1,3],[2,4]\}$ and \\$\barc(N^-)= \{[1,3],[3,4]\}$. 
By \autoref{prop:mdgmrepresentatives}, $[(N^+,N^-)]$ is uniquely represented by $[(M^+,M^-)]$ with $\barc(M^+)=\{[0,4],[2,4]\}$ and $\barc(M^-)=\{[3,4]\}$.
In practice, we will denote this element of $\SPvec$ as  $[0,4]+[2,4]-[3,4]\in \SPvec$.  
If $M,N\in \Pvec$, denote by $M+N$ the element $[(M\oplus N,0)]\in \SPvec$, and denote by $M-N$ the element $[(M,N)]\in \SPvec$.
For an illustration, see \autoref{fig:mdgm-representatives}.
\end{example}

\begin{figure}[t]
    \centering
    \includegraphics[width=.5\linewidth]{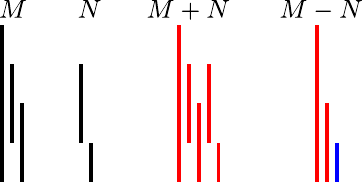}
    \caption{Illustration of the barcodes for $M,N\in \Pvec$ and $M+N,M-N\in \SPvec$.
    For $M+N$ and $M-N$, a red interval is positively signed and a blue interval is negatively signed.
    }
    \label{fig:mdgm-representatives}
\end{figure}

With this notion of signed persistence module in hand, we now use a modified version of the M\"{o}bius inversion formula from \cite{patel2018generalized} to define a meta-diagram:

\begin{definition}
Let $M:\R^2\to \cvec$ be $S$-constructible. 
Define the \emph{horizontal meta-diagram} to be the function $\mdgm_M:\Dgm \to \SPvec$ via the M\"{o}bius inversion formula:
\begin{align*}
\mdgm_{M,x}([s_i,s_j))&:= \mrk_{M,x}([s_i,s_j)) - \mrk_{M,x}([s_i,s_{j+1})) \\ & \hphantom{:= \ }+ \mrk_{M,x}([s_{i-1},s_{j+1})) - \mrk_{M,x}([s_{i-1},s_j)) \\
\mdgm_{M,x}([s_i,\infty)) &:= \mrk_{M,x}([s_i,\infty)) - \mrk_{M,x}([s_{i-1},\infty))
\end{align*}
where $s_0$ is any value $s_0<s_1$ and $s_{n+1}$ is any value $s_{n+1}>s_n$. 
For any other $[s,t)\in \Dgm$, set $\mdgm_{M,x}([s,t)):=0$.
Define the \emph{vertical meta-diagram} by replacing each instance of $x$ above with $y$.
\end{definition}
We henceforth let $\mdgm$ refer to the horizontal meta-diagram of $M$, dropping the subscript when there is no confusion. 
The following M\"{o}bius inversion formula describes the relation between the meta-rank and meta-diagram. It is the direct analogue of \cite[Theorem 4.1]{patel2018generalized}.
\begin{proposition}\label{prop:mdgmmobius}
For $[s,t)\in \Dgm$, we have:
\[\mrk([s,t)) = \sum_{\substack{I\in \Dgm\\ I\supseteq [s,t)}} \mdgm(I)\]
\end{proposition}

\begin{proposition}
\label{prop:directsumsmdgm}
For $M,N:\R^2\to \cvec$, we have:
\[\mdgm_M \oplus \mdgm_N = \mdgm_{M\oplus N},\] 
where $\mdgm_M\oplus \mdgm_N:\Dgm\to \SPvec$ is defined by \begin{align*}
    (\mdgm_M\oplus \mdgm_N)([s,t))&:=[\mdgm_M([s,t))^+\oplus \mdgm_N([s,t))^+,\\ &\hphantom{:=\ [} \mdgm_M([s,t))^-\oplus \mdgm_N([s,t))^-].
\end{align*}
\end{proposition}
 
\autoref{prop:directsumsmdgm} allows us to compute meta-diagrams straightforwardly if we have an indecomposable decomposition of a module.
In particular, by~\autoref{prop:rectanglemdgm}, meta-diagrams are simply computable for rectangle decomposable modules. 

\begin{proposition}
\label{prop:rectanglemdgm}
Suppose $M=\kf^R$ is an $\R^2$-indexed interval module supported on the right-open rectangle $R$, with lower-left corner $(s,t)$ and upper-right corner $(s',t')$. We have:
\[\mdgm_M([a,b)) = 
\begin{cases} 
[t,t') & \mbox{if } \, a=s \, \mbox{ and } \, b=s'; \\ 
0 & otherwise. 
\end{cases}\]
\end{proposition}

\begin{corollary}\label{cor:rectdecompmdgm}
Let $M = \oplus_{R\in \barc(M)} \kf^R$ be rectangle decomposable. 
Then the interval $[t,t')$ appears in $\mdgm([s,s'))$ with multiplicity $n$ if and only if the right-open rectangle with lower-left corner $(s,t)$ and upper right corner $(s',t')$ appears in $\barc(M)$ with multiplicity $n$.
\end{corollary}

\subsection{Equivalence With Rank Decomposition via Rectangles}\label{sec:mrk-rkdecompequiv}

For $M:\R^2\to \cvec$, the rank decomposition by rectangles contains the same information as the rank invariant, which by~\autoref{prop:mrk-rk-equivalence} contains the same information as the meta-rank.
We now show one can directly go from the meta-diagram to the rank decomposition:

\begin{proposition}\label{prop:mdgm-rkdecompequiv}
Let $M:\R^2\to \cvec$ be constructible. 
Define as follows:
\[\mathscr{R}:= \bigcup_{I\in \Dgm} \left( \bigcup_{[a,b)\in \mdgm_M(I)} I\times [a,b)\right),\]

\[\mathscr{S}:= \bigcup_{I\in \Dgm} \left( \bigcup_{-[a,b)\in \mdgm_M(I)} I\times [a,b)\right),\]
where all unions are the multiset union.
Then $(\mathscr{R},\mathscr{S})$ is a rank decomposition for $M$.
\end{proposition}

\begin{proof}
It suffices to show that for all $w_1:=(x_1,y_1)\leq w_2:=(x_2,y_2)\in \R^2$, $\rank_M(w_1,w_2) = \rank_{\kf_\mathscr{R}}(w_1,w_2) - \rank_{\kf_{\mathscr{S}}}(w_1,w_2)$.
Suppose $w_1\leq w_2\in \R^2$ as above.
By \autoref{prop:mrk-rk-equivalence}, 
\[\rank_M(w_1,w_2) = \#[b_i,d_i)\in \mrk_M([x_1,x_2')) \, \, \mbox{ s.t. } \, \, b_i\leq y_1\leq y_2<d_i,\] 
where for notational simplicity, $x_2':=\overline{S}_>(x_2)$.

Now fix $[b,d)$ such that $b\leq y_1\leq y_2<d$.
By \autoref{prop:mdgmmobius}, we have: 
\begin{align*}
   & \#[b,d)\in \mrk_M([x_1,x_2')) = \#[b,d)\in \sum_{\substack{I\in \Dgm\\ I\supseteq [x_1,x_2')}} \mdgm_M(I)\\
    &= \left(\#[b,d)\in \sum_{\substack{I\in \Dgm\\ I\supseteq [x_1,x_2')}} \mdgm^+_M(I)\right) - \left(\#[b,d)\in \sum_{\substack{I\in \Dgm\\ I\supseteq [x_1,x_2')}} \mdgm^-_M(I)\right)
\end{align*}
By \autoref{prop:rectanglemdgm} and \autoref{cor:rectdecompmdgm}, the term $\#[b,d)\in \sum\limits_{\substack{I\in \Dgm\\ I\supseteq [x_1,x_2')}} \mdgm^+(I)$ is the number of times $I\times [b,d)$ appears in $\mathscr{R}$ across all $I\supseteq [x_1,x_2')$, and the term $\#[b,d)\in \sum\limits_{\substack{I\in \Dgm\\ I\supseteq [x_1,x_2')}} \mdgm^-(I)$ is the number of times $I\times [b,d)$ appears in $\mathscr{S}$ across all $I\supseteq [x_1,x_2')$.

Thus, we see that $\rank_M(w_1,w_2)$ is equal to the number of rectangles in $\mathscr{R}$ containing $w_1$ and $w_2$ minus the number of rectangles in $\mathscr{S}$ containing $w_1$ and $w_2$. 
From the definition of rectangle module and the fact that $\rank$ commutes with direct sums, the first term is $\rank(\kf_\mathscr{R})(w_1,w_2)$ and the second term is $\rank(\kf_\mathscr{S})(w_1,w_2)$, and so we get: 
\[ \rank_M(w_1,w_2) = \rank_{\kf_\mathscr{R}}(w_1,w_2) - \rank_{\kf_\mathscr{S}}(w_1,w_2)\qedhere\]
\qedhere
\end{proof}

\subsection{Stability of Meta-Diagrams}
We now show a stability result for meta-diagrams.
We need to modify the notion of erosion distance to do so, as meta-diagrams have negatively signed parts.
We proceed by adding the positive part of one meta-diagram to the negative part of the other.
This idea stems from Betthauser et al.'s work \cite{betthauser2022graded}, and was also used in the stability of rank decompositions in \cite{botnan2022signed}.

\begin{definition}
For $M,N:\R^2\to \cvec$, define $\PN(M,N):\Dgm\to \cvec$ as
\[\PN(M,N)([s,t)):= \mdgm_M^+([s,t)) + \mdgm_N^-([s,t))\]
\end{definition}

$\PN(M,N)$ is a non-negatively signed 1-parameter persistence module for all $[s,t)\in \Dgm$, allowing us to make use of the previous notion of $\preceq_\epsilon$ (\autoref{def:compare}) to define an erosion distance for meta-diagrams. 
Unlike meta-ranks which have a continuous support, a meta-diagram is only supported on $(\overline{S})^2$ for some finite $S\subset \R$.
As a result, we first modify the notion of erosion distance to fit the discrete setting.

Define maps $\overline{S}_\geq,\overline{S}_\leq:\R\cup\{\infty\}\to \overline{S}$ by $\overline{S}_\geq(x):=\min\{s\in \overline{S} \,|\, x\geq s\}$ and $\overline{S}_\leq(x):=\max\{s\in \overline{S} \,|\, x\leq s\}$, or some value less than $s_1$ if this set is empty.
We say $S$ is \emph{evenly-spaced} if there exists $c\in \R$ such that $s_{i+1}-s_i=c$ for all $1\leq i\leq n-1$. 
In the following, fix an evenly-spaced finite $S\subset \R$.

\begin{definition}
\label{def:dEmdgm}
For $S$-constructible $M,N:\R^2\to \cvec$, define the erosion distance: 
\begin{align*}
\dES(\mdgm_M,\mdgm_N):=&\inf \{\epsilon\geq 0 \, | \, \forall s\leq t\in \overline{S}, \\
\, &\PN(M,N)([\overline{S}_\leq(s-\epsilon),\overline{S}_\geq(s+\epsilon))^\epsilon \preceq_{2\epsilon} \PN(N,M)([s,t)) \, \mathrm{and} \\
&\PN(N,M)([\overline{S}_\leq(s-\epsilon),\overline{S}_\geq(s+\epsilon))^\epsilon \preceq_{2\epsilon} \PN(M,N)([s,t))\}
\end{align*}
\end{definition}
We have the following stability result for meta-diagrams, 
\begin{theorem}\label{thm:mdgmdEstability}
For $S$-constructible $M,N:\R^2\to \cvec$, with $S$ evenly-spaced, we have
\[\dES(\mdgm_M,\mdgm_N)\leq \dI(M,N).\]
\end{theorem}
For details and a stability result when $S$ is not evenly-spaced, see \autoref{sec:meta-diagram-details}.

\section{Algorithms}
\label{sec:algorithms}

In this section, we provide algorithms for computing meta-ranks and meta-diagrams.
The input to these algorithms is a simplex-wise bifiltration: 
\begin{definition}
\label{def:bifiltration}
Let $n \in \Z$, and $[n]$ denote the poset $\{1,\ldots,n\}$ with the usual order.
Let $K$ be a simplicial complex, and $\mathrm{sub}(K)$ denote all subsets of $K$ which are themselves simplicial complexes.
A filtration is a function $F:[n]\to \mathrm{sub}(K)$ such that for $a\leq b$, $F(a)\subseteq F(b)$.
We say a filtration is simplex-wise if for all $1\leq a\leq n-1$, either $F(a+1)=F(a)$ or $F(a+1) = F(a) \cup \{\sigma\}$ for some $\sigma \in K\setminus F(a)$.
In the latter case, we denote this with $F(a)\xrightarrow{\sigma} F(a+1)$.
We say $\sigma\in \mathrm{sub}(K)$ arrives at $a$ if $\sigma\in F(a)$ and $\sigma\notin F(a-1)$.

Define $P_n:=[n]\times [n]$ equipped with the product order.
A bifiltration is a function $F:P_n\to \mathrm{sub}(K)$.
We say a bifiltration is simplex-wise if for all $(a,b)\in P_n$, for $(x,y)=(a+1,b)$ or $(a,b+1)$, if $(x,y)\in P_n$, then either $F((x,y))=F((a,b))$, or $F((a,b))\xrightarrow{\sigma} F((x,y))$ for some $\sigma\notin F((a,b))$.
\end{definition}

Applying homology to a bifiltration yields a bimodule defined on $P_n$. 
Our theoretical background in previous sections focused on the case of bimodules defined over $\R^2$.
The same ideas and major results follow similarly for a module defined over $P_n$.
We quickly highlight the differences in definitions when working with modules defined on $P_n$.
The following definitions are re-phrasings of the horizontal meta-rank and horizontal meta-diagram for modules indexed over $P_n$, but as before, the statements are directly analogous in the vertical setting.
Let $\Int([n])$ refer to all intervals of $[n]$, which consists of $\{[a,b] \, | \, a\leq b, \, a,b\in [n]\}$.
\begin{definition}\label{def:intmrk}
For $M:P_n\to \cvec$, define the meta-rank, $\mrk_{M}:\Int([n])\to \Pvec$ by
\[\mrk_M([s,t]):= [\im(\phi_x(s\leq t))]\]
\end{definition}

\begin{definition}\label{def:intmdgm}
For $M:P_n\to \cvec$, define the meta-diagram, $\mdgm_{M}:\Int([n])\to \SPvec$ as follows:
if $1<s\leq t<n$, define:
\begin{align*}
    \mdgm_{M}([s,t])&:= \mrk_{M}([s,t]) - \mrk_{M}([s,t+1]) \\
    &\ \ \ \ + \mrk_{M}([s-1,t+1]) - \mrk_{M}([s-1,t]),\\
    \mdgm_{M}([s,n])&:= \mrk_{M}([s,n]) - \mrk_{M}([s-1,n]),\\
    \mdgm_{M}([1,t])&:= \mrk_M([1,t]) - \mrk_M([1,t+1]), \ \mathrm{and}\\
    \mdgm_M([1,n])&:= \mrk_M([1,n]).
\end{align*}
\end{definition}

\subsection{Overview of the Algorithm}

Henceforth, assume $F:P_n\to \mathrm{sub}(K)$ is a simplex-wise bifiltration, and $M$ is the result of applying homology to $F$.
Our algorithm to compute the meta-rank relies on the vineyards algorithm from \cite{cohen2006vines}.
The algorithm starts with $F$ as the input.
Define $\gamma_1$ to be the path in $P$ going from $(1,1)\to(1,n)\to(n,n)$, i.e. the path along the top-left boundary of $P$.
We compute the $D=RU$ decomposition for the interval decomposition of the persistence module given by the 1-parameter filtration found by slicing $F$ over $\gamma_1$, which we denote $F_{\gamma_1}$.
This decomposition gives us all the persistence intervals and persistence pairs $(\sigma_i,\sigma_j)$ and unpaired simplices corresponding to each interval, the former corresponding to a finite interval and the latter an infinite interval.
To simplify notation, for every unpaired simplex corresponding to an infinite interval, we pair it with an implicit simplex arriving in an extended $F$ at $(n+1,n+1)$.
We store the persistence intervals in an ordered list, which we denote \texttt{intervals}.
All intervals in \texttt{intervals} restricted to $[1,n]$ constitute together the 1-parameter  persistence module $M_x^1$, which is precisely $\mrk_M([1,1])$.
We then store $\mrk_M([1,1])$ as a list, with the same ordering as \texttt{intervals}, leaving an empty placeholder whenever an interval does not intersect $[1,n]$.

We sweep $\gamma$ through $P$, over one square at a time, going down through the first column, until we reach $\gamma_2$, the path $(1,1)\to (2,1)\to(2,n)\to(n,n)$.
From there, we repeat the process column-by-column until we reach $\gamma_n$, the path $(1,1)\to(n,1)\to(n,n)$; see \autoref{fig:big-algorithm}.

\begin{figure}[t]
    \centering
    \includegraphics[width=.3\linewidth]{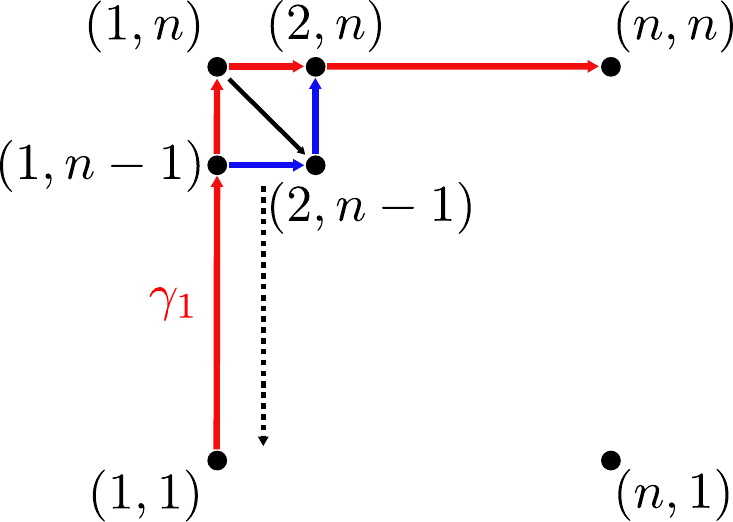}
    \hspace{6mm}
    \includegraphics[width=.26\linewidth]{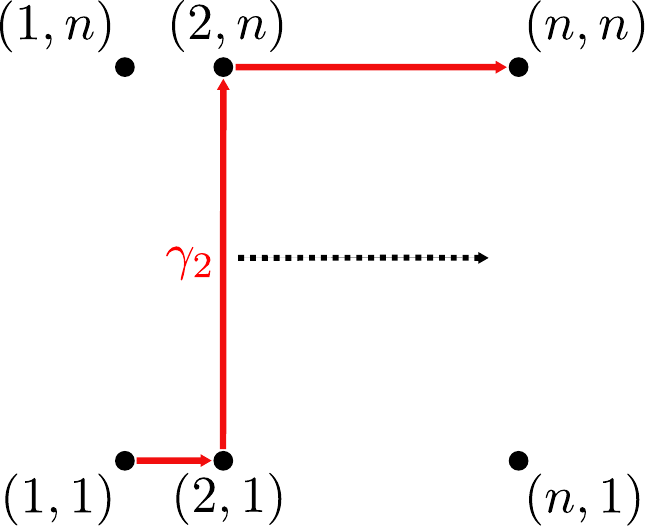}
    \hspace{6mm}
    \includegraphics[width=.26\linewidth]{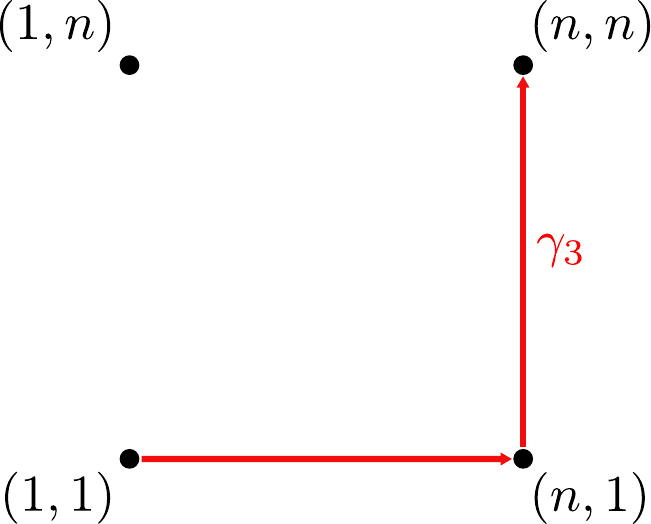}
    \caption{We start with $\gamma_1$ on the left, and then push $\gamma_1$ through the square to track along the lower-right corner of the square (in blue). 
    We repeat this process, descending down each square in the first column until we reach $\gamma_2$ (middle).
    Then we repeat this process column-by-column until we've reached $\gamma_n$ (right).} 
    \label{fig:big-algorithm}
\end{figure}

After each change of a single vertex in our intermediary paths $\gamma$ stemming from swapping the upper-left boundary of a single square to the lower-right one, the resulting filtration $F_\gamma$ either remains the same, or changes in one of the ways illustrated in \autoref{fig:square-possibilities}.

\begin{figure}[t]
    \centering
    \includegraphics{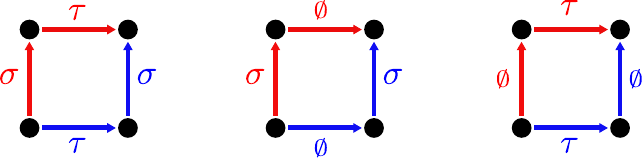}
    \caption{Three possible ways in which $F_\gamma$ can change via being pushed through a one-by-one square.
    In our algorithm, $\gamma$ always starts along the red path, then shifts to the blue path.}
    \label{fig:square-possibilities}
\end{figure}

After passing through each square, we update each interval in \texttt{intervals} in-place.
If $F_\gamma$ remains the same, then there is no change to \texttt{intervals}.
If $F_\gamma$ changes by altering the arrival time of a single simplex, then the pairings do not change, and the interval corresponding to the shifted simplex either extends by one or shrinks by one.
If a transposition occurs, see \autoref{fig:square-possibilities} (left), then we use the transposition update process from the vineyards algorithm.

If we start at $\gamma_1$, then when we reach $F_{\gamma_2}$, we can restrict each interval in \texttt{intervals} to $[2,n+1]$ and shift it back down one, and this corresponds to $\mrk_M([2,2])$, which we store using the same rules as we did with $\mrk_M([1,1])$.

Since we are storing all intervals in meta-ranks in this ordered fashion, we can take any interval in $\mrk_M([2,2])$, and see where it came from in $\mrk_M([1,1])$, which would be the interval stored at the same index in both lists.
By taking the intersection, we get the corresponding interval which we put into this location in the list $\mrk_M([1,2])$.
We repeat the process of modifying $\gamma$ one vertex at a time to get the paths $\gamma_i$ from $(1,1)\to (i,1)\to (i,n)\to (n,n)$ as above, updating \texttt{intervals} and getting $\mrk_M([i,i])$ by taking appropriate intersections and shifts.
Since every list of intervals we store maintains this ordering, we can take any interval in $\mrk_M([i,i])$, and see the corresponding interval it was previously (if any) in $\mrk_M([k,i-1])$ for all $1\leq k\leq i-1$.
Then by intersecting the interval in $\mrk_M([i,i])$ with its corresponding interval in $\mrk_M([k,i-1])$, we get a new corresponding interval in $\mrk_M([k,i])$.
We repeat this process iteratively with $i$ going from $1$ to $n$, which at the end computes all of $\mrk_M:\Int([n])\to \Pvec$.

We now describe what can happen to the intervals as we pass over a single square in which a transposition occurs, swapping $\sigma_i$ and $\sigma_{i+1}$.
From the analysis in \cite[Section 3]{cohen2006vines}, if the pairing function changes, then the intervals themselves do not change. 
If the pairing function remains the same, then two of the persistence intervals will change.
Suppose $\sigma_i$ is paired with $\tau_i$ and $\sigma_{i+1}$ is paired with $\tau_{i+1}$.
There are four possibilities, see \autoref{fig:same-pair-updates}.

\begin{figure}[t]
    \centering
    \includegraphics{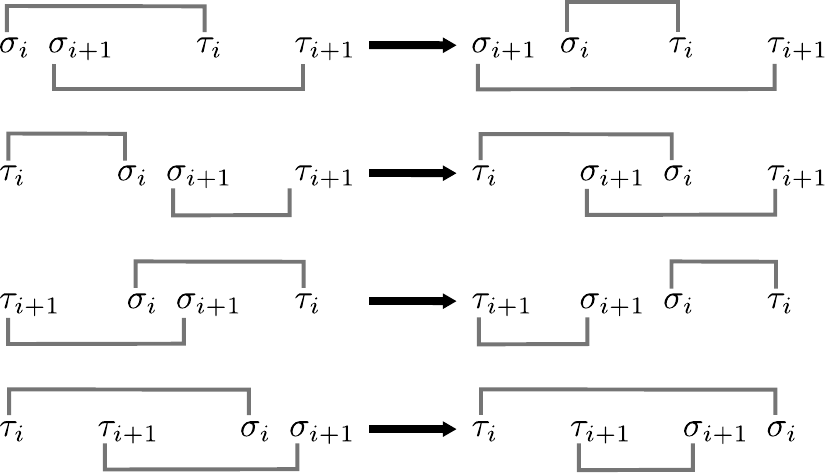}
    \caption{Four cases in which intervals change after a transposition.
    Observe that in each case, both intervals change, and this change is in exactly one coordinate.
    }
    \label{fig:same-pair-updates}
\end{figure}

\begin{algorithm}[t]
\caption{{\sc MetaRank}($F$)}
\label{alg:mrk-main-alg}
\begin{itemize}
    \item\label{item:step 1} Step 1. Compute $D=RU$ for $F_{\gamma_1}$, getting the ordered list \texttt{intervals} and the pairing for each interval.
    \item\label{item:step 2} Step 2. for each interval in \texttt{intervals}, intersect the interval with $[1,n]$, and store the result in the ordered list $\mrk([1,1])$.
    \item\label{item:step 3} Step 3. For $i:=1$ to $n-1$, do
    \begin{itemize}
        \item\label{item:step 3.1} Step 3.1. For $j:=n$ down to $2$, do
        \begin{itemize}
            \item update $D$, $R$, $U$, and \texttt{intervals} via the vineyards algorithm, as $\gamma$ sweeps through the square with upper-left corner $(i,j)$ and lower-right corner $(i+1,j-1)$.
        \end{itemize}
        \item\label{item:step 3.2} Step 3.2. For each interval in \texttt{intervals}, shift the interval down by $i-1$, and intersect the interval with $[1,n]$, storing the result in the ordered list $\mrk([i,i])$.
        \item\label{item:step 3.3} Step 3.3. For $k:=1$ to $i-1$, do
        \begin{itemize}
            \item For each interval in $\mrk([i,i])$, intersect with the corresponding interval in $\mrk([k,i-1])$.
            Store this intersection in the ordered list $\mrk([k,i])$.
        \end{itemize}
    \end{itemize}
\end{itemize}
\end{algorithm}

We describe the algorithm in \autoref{alg:mrk-main-alg}.
The output of \autoref{alg:mrk-main-alg} will be $\mrk_M$, stored as a collection of lists of the barcodes $\mrk_M([s,t])$ for all $s\leq t\in [n]$.

We now prove the correctness of \autoref{alg:mrk-main-alg}.

\begin{proposition}
\label{prop:mrkalg1}
For $i\in [n]$, $\mrk_M([i,i])$ is found by taking each interval in the barcode for $F_{\gamma_i}$, shifting it down by $i-1$, and then taking the intersection with $[1,n]$.
\end{proposition}

\begin{proposition}
\label{prop:mrkalg2}
Let $1<i\leq n$, and suppose we know $\mrk_M([i,i])$ and $\mrk_M([k,i-1])$ for all $1\leq k\leq i-1$, and that these lists of intervals are stored in the ordered fashion previously described.
From this information, we can compute $\mrk_M([k,i])$.
\end{proposition}

\begin{theorem}
\label{thm:mrkalg}
\autoref{alg:mrk-main-alg} correctly computes the meta-rank for the bimodule $M$ induced by homology of the input bifiltration $F$, and runs in time $O(n^3)$.
As a result, the number of rectangles in the rank decomposition for $M$ is also $O(n^3)$.
\end{theorem}
\begin{proof}
By \autoref{prop:mrkalg1}, we can compute $\mrk_M([1,1])$, and further $\mrk_M([i,i])$ for all $i\in [n]$.
Then we can use \autoref{prop:mrkalg2} 
iteratively to fill in $\mrk_M([k,i])$ for all $1\leq k<i\leq n$, and we are done. 

For the runtime analysis, first observe that the initial $D=RU$ computation in Step 1 takes $O(n^3)$ time, and \texttt{intervals} can be
computed from the decomposition in linear time.
The loop in Step 2 also takes linear time, as the size of \texttt{intervals} is $O(n)$ which is fixed throughout.
Step 3 consists of a for loop with $O(n)$ iterations.
Step 3.1 consists of a for loop with $O(n)$ iterations, and each loop inside performs an update over a square using the vineyards approach.
A single update takes $O(n)$ time in the worst case, so Step 3.1 takes $O(n^2)$ time.
Step 3.2 runs in linear time for the same reason as Step 2.
Step 3.3 consists of a for loop with $O(n)$ iterations, with each iteration taking $O(n)$ operations as the size of each $\mrk_M([k,i])$ is the same as \texttt{intervals}.
Hence, Step 3.3 has total runtime $O(n^2)$.
Thus, each loop in Step 3 consists of substeps that run in $O(n^2)$ time, $O(n)$ time, and $O(n^2)$ time respectively, incurring a total cost of $O(n^3)$ over $O(n)$ iterations.

To summarize, we have a step with $O(n^3)$ cost, followed by a step with $O(n)$ cost, followed by a step with $O(n^3)$ cost, so the algorithm runs in $O(n^3)$ time.

By Definition \ref{def:intmdgm}, we can compute $\mdgm_M$ from $\mrk_M$ in $O(n^3)$ time, implying the number of non-zero intervals in $\mdgm_M$ is $O(n^3)$.
By Proposition \ref{prop:mdgm-rkdecompequiv}, each non-zero interval in $\mdgm_M$ corresponds uniquely to a single rectangle in the rank decomposition of $M$, and so the number of such rectangles is likewise $O(n^3)$.
\end{proof}

\section{Discussion}
\label{sec:discussion} 

We conclude with some open questions.
First, we would like to extend our approach to the $d$-parameter setting. 
We expect that a proper extension would satisfy relationships with the rank invariant and rank decompositions similar to \autoref{prop:mrk-rk-equivalence} and \autoref{prop:mdgm-rkdecompequiv}. 
Such an extension would also lead to a ``recursive'' formulation of the persistence diagram of diagrams illustrated in~\autoref{fig:diagram-of-diagrams}.
Next, \autoref{thm:mrkalg} implies that the number of rectangles needed in a rank decomposition for a bimodule is bounded above by $O(n^3)$.
It is not known whether this bound is tight. 
Lastly, there have been multiple recent works that use algorithmic ideas from $1$-parameter persistence to compute invariants in the multiparameter setting \cite{DKM22,hickok2022computing,morozov2021output}. 
We wish to explore in what ways these approaches can create new algorithms or improve upon existing ones for computing the invariants of multi-parameter persistence modules.

\bibliography{meta-refs}

\appendix

\section{Detailed Proofs for Meta-Rank}
\label{sec:meta-rank-details}

\begin{proof}[Proof of \autoref{prop:mrk-rk-equivalence}]
We start by showing that 
\begin{equation}\label{eq:prop12eq}
    \rank_M((s,y),(t,y')) = \#[b_i,d_i)\in \mrk_M([s,\overline{S}_>(t))) \,  \, s.t. \, \, b_i\leq y\leq y'<d_i.
\end{equation}
From the commutativity conditions on persistence modules, we have: 
\[\varphi((s,y)\leq (t,y')) = \varphi((t,y)\leq (t,y'))\circ \varphi((s,y)\leq (t,y)),\] 
and observe that $\varphi((s,y)\leq (t,y)) = \phi_x(s\leq t)\vert_{M((s,y))}$.
From \autoref{def:mrk}, one can check that $\mrk_M([s,\overline{S}_>(t))) = [\im(\phi_x(s\leq t))]$.
For simplified notations, let $h:=\varphi((x,y)\leq (t,y'))$, $f:=\varphi((t,y)\leq (t,y'))$,  $g:=\varphi((s,y)\leq (t,y))$, and $N:=\im(\phi_x(s\leq t))$. 
We have a commutative diagram: 

\[\begin{tikzpicture}
\draw (1,0) -- (1,5);
\draw (6,0) -- (6,5);

\node at (1,-.5) {$s$};
\node at (6,-.5) {$t$};
\node at (1,5.2) {$M_x^s$};
\node at (6,5.2) {$\im(\phi_x(s\leq t))$};

\node at (4.2,2) {$\circlearrowleft$};

\draw[-to] (1,1) -- (5.9,1);
\draw[-to] (1,1) -- (5.9,3.94);
\draw[-to] (6,1) -- (6,3.9);

\node at (1,1) {\textbullet};
\node at (0.5,0.8) {$(s,y)$};
\node at (6,1) {\textbullet};
\node at (6.5,0.8) {$(t,y)$};
\node at (6,4) {\textbullet};
\node at (6.5,4.3) {$(t,y')$};

\node at (3.5, .8) {$g$};
\node at (3.4, 2.8) {$h$};
\node at (6.2, 2.5) {$f$};

\end{tikzpicture}\]

We know from linear algebra that $\rank(h) = \rank(g)-\dim(\ker \,f\cap \im(g))$.
As noted, $N(y) = \im(g)$, so $\rank(g) = \dim(N(y))$.
It is immediate that $\dim(N(y))$ is equal to the number of intervals in $\barc(N)$ which contains $y$.
Furthermore, by the commutativity of internal morphisms of $M$, we have that $f\vert_{\im(g)}$ is exactly the internal morphism $\varphi_N(y\leq y')$.
From this and the rank-nullity theorem, we have:
\[\dim(N(y)) = \dim(\im(\varphi_N(y\leq y'))) + \dim(\ker \varphi_N(y\leq y’)).\]
As $\rank(g) = \dim(N(y))$ and $\dim(\ker \, \varphi_N(y\leq y') = \dim(\ker \, f\vert_{\im(g)}) = \dim(\ker \, f\cap \im(g))$, we find $\rank(h) = \dim(\im(\varphi_N(y\leq y')))$.
$\rank(h)$ is precisely $\rank_N(y\leq y')$, which is well-known to be the number of bars in $\barc(N)$ containing $[y,y']$.
As a result, we can compute $\rank_M$ from $\mrk_M$.

Now we show the other claim, that we can compute $\mrk_M$ from $\rank_M$.
By the definition of constructible bimodule, any critical point of the rank function is of the form $((s_a,s_b),(s_c,s_d))$ for some $s_a,s_b,s_c,s_d\in S$.
Hence, from the rank function of $M$, we can determine the minimal $S$ on which $M$ is $S$-constructible.
Assume that $S$ in the remainder is the minimal $S$ on which $M$ is $S$-constructible.

Fix some $[s,t)\in \Dgm$, and fix an interval $[y,y')\in \Dgm$.
We show that from $\rank_M$, we can determine the multiplicity of the interval $[y,y')$ in $\barc(\mrk_M([s,t)))$, denoted as $\#[y,y')$.
If $s<s_1$, then by \autoref{def:dgm} we have $\mrk_M([s,t)) = 0$.
Thus, assume $s\geq s_1$ and define $S_<,S_{\leq}:\R_{\geq s_1}\cup\{\infty\}\to S$ by $S_<(t):= \max\{s\in S \, | \, s<t\}$ and $S_\leq(t):= \max\{s\in S \, | \, s\leq t\}$. 

As a consequence of $M$ being $S$-constructible, all intervals in $\barc(\mrk_M([s,t)))$ are of the form $[s_i,s_j)$ or $[s_i,\infty)$ for some $s_i,s_j\in S$. 
If $[y,y') = [s_i,s_j)$, then by the well-known inclusion-exclusion formula in 1-parameter persistence and the formula in 
\autoref{eq:prop12eq}, we can compute:
\begin{align*}
    \#[s_i,s_j) &= \rank_M((S_\leq(s),s_i),(S_<(t),s_{j-1})) -
    \rank_M((S_\leq(s),s_i),(S_<(t),s_j))\\
    &\hphantom{=}+ \rank_M((S_\leq(s),s_{i-1}),(S_<(t),s_j)) - \rank_M((S_\leq(s),s_{i-1}),(S_<(t),s_{j-1})),
\end{align*} 
where $s_{n+1}$ is any value $s_{n+1}> s_n$, and $s_0$ is any value $s_0<s_1$.
If $[y,y') = [s_i,\infty)$, then analogously we can compute:
\[ \#[s_i,\infty) = \rank_M((S_\leq(s),s_i),(S_<(t),s_n)) - \rank_M((S_\leq(s),s_{i-1}),(S_<(t),s_n)).\]

Therefore, for any $[s,t)\in \Dgm$, and $[y,y')\in \Dgm$, we can compute the multiplicity of $[y,y')\in \mrk_M([s,t))$ from $\rank_M$, and so we can compute all of $\mrk_M$ from $\rank_M$.\qedhere 
\end{proof}

\begin{proof}[Proof of Proposition \ref{prop:mrk-pseudometric}]
Symmetry is clear from the definition of $\dE$.
It remains to check the triangle inequality.
Suppose $M,N,L:\R^2\to \cvec$ are such that $\forall I\in \Dgm$, $\mrk_{M}(I^{\epsilon_1}_{-\epsilon_1})^{\epsilon_1}\preceq_{2\epsilon_1} \mrk_{N}(I)$ and $\mrk_{N}(I^{\epsilon_1}_{-\epsilon_1})^{\epsilon_1}\preceq_{2\epsilon_1} \mrk_{M}(I)$. 
Also, suppose $\forall I\in \Dgm$, 
$\mrk_{N}(I^{\epsilon_2}_{-\epsilon_2})^{\epsilon_2}\preceq_{2\epsilon_2} \mrk_{L}(I)$ and $\mrk_{L}(I^{\epsilon_2}_{-\epsilon_2})^{\epsilon_2}\preceq_{2\epsilon_2} \mrk_{N}(I)$.

Fix any $I\in \Dgm$.
It is clear that $(I^{\epsilon_1}_{-\epsilon_1})^{\epsilon_2}_{-\epsilon_2} = I^{\epsilon_1+\epsilon_2}_{-\epsilon_1-\epsilon_2}$, and so we have:
\[\mrk_{M}(I^{\epsilon_1+\epsilon_2}_{-\epsilon_1-\epsilon_2})^{\epsilon_1+\epsilon_2}\preceq_{2(\epsilon_1+\epsilon_2)} \preceq \mrk_{N}(I^{\epsilon_2}_{-\epsilon_2})^{\epsilon_2} \preceq_{2\epsilon_2} \mrk_L(I)\]
and similarly with the roles of $M$ and $L$ reversed.
Hence, $\dE(\mrk_M,\mrk_L)\leq \epsilon_1+\epsilon_2$, as desired.
\end{proof}

The following Lemma is useful in the Proof of \autoref{thm:mrk-stability}:
\begin{lemma}\label{lemma:truncated-barcode}
Let $M:\R\to \cvec$ be a persistence module, with barcode $\barc(M)$, and let $\epsilon > 0$. Define $M[\epsilon:]:\R\to \cvec$ as follows: for $a\leq b\in \R$, 
\[M[\epsilon:](a):=\{x\in M(a) \, | \, \exists x'\in M(a-\epsilon) \, s.t.\ \varphi_M(a-\epsilon\leq a)(x')\} \]
\[M[\epsilon:](a\leq b):= M(a\leq b)\vert_{M[\epsilon:](a\leq b)}\]
Then $M[\epsilon:]:\R\to \cvec$ is a well-defined persistence module, and $\barc_\epsilon(M) = \barc(M[\epsilon:])$.
\end{lemma}
\begin{proof}[Proof of Lemma \ref{lemma:truncated-barcode}]
Let $M = \oplus_{I\in \mathcal{I}} \kf^I$, and $\{e^t_I\}_{I\in \mathcal{I}}^{t\in \R}$ be such that $\{e^t_I\}_{I\in \mathcal{I}}$ is a basis for $M(t)$ for all $t$.
Further, require $e^t_i\neq 0 \iff t\in I$, and $\varphi_M(s\leq t)(e^s_I) = e^t_I$.
The intuition is that each element $e^t_I\in M(t)$ is either 0 or a basis for the summand $\kf^I(t)$ of $M(t)$. 
We call such a set $\{e^t_I\}_{I\in \mathcal{I}}^{t\in \R}$ a \emph{persistence basis}. 
From the definitions, $e^t_I\neq 0\in M[\epsilon:](t)$ if and only if $e^t_I\neq 0\in M(t)$ and $e^{t-\epsilon}_I\neq 0\in M(t-\epsilon)$. 
Thus, if $\{e^t_{I_1},\ldots, e^t_{I_n}\}$ is a basis for $M(t)$, then a subset of these will be a basis for $M[\epsilon:](t)$. 

For $a\leq b\in \R$, to see that $\varphi_M(a\leq b)\vert_{M[\epsilon:](a)}$ maps $M[\epsilon:](a)$ into $M[\epsilon:](b)$, we can consider the mapping on basis elements. 
If $e^a_I\neq 0\in M[\epsilon:](a)$, then $e^{a-\epsilon}_I \neq 0 \in M(a-\epsilon)$.
It follows that
\[\varphi_M(a\leq b)(e^a_I) = \, \varphi_M(b-\epsilon\leq b)(\varphi_M(a-\epsilon\leq b-\epsilon)(e^{a-\epsilon}I)) \]
and so $\varphi_M(a\leq b)\vert_{M[\epsilon:](a)}(M[\epsilon:](a)\subseteq M[\epsilon:](b)$, and so $M[\epsilon:]$ is a well-defined persistence module.

Now we show that $\barc_\epsilon(M) = \barc(M[\epsilon:])$. 
Suppose $I=[s,t)\in \barc_\epsilon(M)$. 
Then $I$ corresponds uniquely to an interval $[s-\epsilon,t)\in \barc(M)$. 
Suppose $M=\oplus_{I\in \mathcal{I}} k^I$.
This interval in $\barc(M)$ corresponds uniquely to a specific sequence, namely for a fixed $I\in \mathcal{I}$, a sequence of nonzero elements  $\{e^a_I\}_{a\in [s-\epsilon,t)}$, with $e^a_I\in M(a)=0$ for $a\notin [s-\epsilon,t)$.
It is straightforward to check that if $\{e^t_I\}_{i\in \mathcal{I}}^{t\in \R}$ is a persistence basis for $M$, then $\{e'^t_I\}_{I\in \mathcal{I}}^{t\in \R}$ is a persistence vector basis for $M[\epsilon:]$, where $e'^t_I= e^t_I\in M[\epsilon:](t) \iff e^t_I\neq 0\in M(t)$ and $e^{t-\epsilon}_I\neq 0\in M(t-\epsilon)$.
Otherwise, $e'^t_I:=0$.
This means $e'^a_I \neq 0\in M(a)$ if and only if $e^a_I\neq 0\in M(a)$ and $e_I^{a-\epsilon}\neq 0\in M(a-\epsilon) \iff a\in I$. 
Thus, this sequence $\{e^a_I\}_{a\in [s,t)}$ corresponds uniquely to an interval $[s,t)\in \barc(M[\epsilon:])$. 
So we have every interval $[s,t)\in \barc_\epsilon(M)$ corresponds uniquely to an interval in $\barc(M[\epsilon:])$, so $\barc_\epsilon(M)\subseteq \barc(M[\epsilon:])$. 

To see the reverse containment, if $I=[s,t)$ is an interval in $\barc(M[\epsilon:])$, then we can reverse the previous argument to see that this corresponds uniquely to a sequence of nonzero elements $\{e^a_I\}_{a\in [s-\epsilon,t)}$ in $M$. 
This corresponds uniquely to an interval $[s-\epsilon,t)$ in $\barc(M)$, which corresponds to an interval $[s,t)\in \barc_\epsilon(M)$. 
Hence, $\barc(M[\epsilon:])\subseteq \barc_\epsilon(M)$, and so $\barc(M[\epsilon:])=\barc_\epsilon(M)$, as desired.
\end{proof}

\begin{proof}[Proof of Theorem \ref{thm:mrk-stability}]
Suppose $\epsilon\geq 0$ and $f:M\to N^{\beps}$ and $g:N\to M^{\beps}$ are an interleaving pair with $\beps=(\epsilon,\epsilon)$. 
Fix $S$ so that $M$ and $N$ are both $S$-constructible.
Let $I=[s,t)\in \Dgm$. 
Assume initially that $t\notin S$ and $t+\epsilon\notin S$, these cases will be dealt with at the end.

By the definition of constructibility, we can replace $[s,\infty)$ with $[s,c)$ for some $c\geq s_n$ (recall $s_n$ is the maximal element in $S$), 
so we will show the result under the assumption $[s,t)\in \Dgm$, with $t<\infty,\, t\notin S,$ and $t+\epsilon\notin S$.

Under our assumption, $\mrk_M(I)=[\im(\phi_x^M(s\leq t))],$ and $\mrk_N(I^\epsilon_{-\epsilon}) = [\im(\phi_x^N(s-\epsilon\leq t+\epsilon))]$.
Denote by $f'$ the restriction of $f$ to $\im(\phi_x^M(s\leq t))$.
Note that $f'$ maps into $N^{t+\epsilon}_x$.
We claim that $\im(\phi_x^N(s-\epsilon\leq t+\epsilon))[2\epsilon:]\subseteq \im(f')$.

To see this, let $a\in \R$, and let $x\in \im(\phi_x^N(s-\epsilon\leq t+\epsilon))^\epsilon[2\epsilon:](a)$.
By definition, this means there exists $x'\in \im(\phi_x^N(s-\epsilon\leq t+\epsilon))(a-\epsilon)$ such that $\varphi_N((t+\epsilon,a-\epsilon)\leq (t+\epsilon,a+\epsilon))(x') = x$.
Further, there is an $x''\in N_x^{s-\epsilon}(a-\epsilon)$ such that $\varphi_N((s-\epsilon,a-\epsilon)\leq (t+\epsilon,a-\epsilon))(x'') = x'$.
Set $y:= \varphi_M((s,a)\leq (t,a))(g(x''))$.
From this definition, it is clear that $y\in \im(\phi_x^M(s\leq t))$.
By the interleaving condition between $f$ and $g$, we have:
\begin{align*}
    f'(y) &= f'(\varphi_M((s,a)\leq (t,a))(g(x''))\\
    &= f'(g(\varphi_N((s-\epsilon,a-\epsilon)\leq (t-\epsilon,a-\epsilon))(x'')\\
    &= \varphi_N((t-\epsilon,a-\epsilon)\leq (t+\epsilon,a+\epsilon))(\varphi_N((s-\epsilon,a-\epsilon)\leq(t-\epsilon,a-\epsilon))(x''))\\
    &= \varphi_N((t+\epsilon,a-\epsilon)\leq (t+\epsilon,a+\epsilon))(\varphi_N((s-\epsilon,a-\epsilon)\leq (t+\epsilon,a-\epsilon))(x''))\\
    &= \varphi_N((t+\epsilon,a-\epsilon)\leq (t+\epsilon,a+\epsilon))(x') = x
\end{align*}
As a result, we have a surjective map $f':\im(\phi_x^M(s\leq t))\to \im(f')$, and an injective inclusion of persistence modules $\iota:\im(\phi_x^N(s-\epsilon\leq t+\epsilon))[2\epsilon:]\hookrightarrow \im(f')$.
By \cite{bauer2015induced} these maps induce injective maps on barcodes $\chi_{f'}:\barc(\im(f'))\hookrightarrow \barc(\mrk_M([s,t)))$ and $\chi_{\iota'}:\barc(\mrk_N([s-\epsilon,t+\epsilon)))^\epsilon[2\epsilon:])\hookrightarrow \barc(\im(f'))$.
By \autoref{lemma:truncated-barcode}, we can view $\chi_\iota$ as a map with domain $\barc_{2\epsilon}(\mrk_N([s-\epsilon,t+\epsilon))^\epsilon)$.

Define $\chi:=\chi_{f'}\circ \chi_\iota:\barc_{2\epsilon}(\mrk_N([s-\epsilon,t+\epsilon))^\epsilon)\to \mrk_M([s,t))$.
This is injective as it is a composition of injections.
For all $J\in \barc_{2\epsilon}(\mrk_N([s-\epsilon,t+\epsilon))^\epsilon)$, we have $\chi(J)=\chi_{f'}(\chi_\iota(J))\subseteq \chi_\iota(J)\subseteq J$. Thus, $\mrk_N(I^\epsilon_{-\epsilon})\preceq_{2\epsilon} \mrk_M(I)$.
The argument is symmetric when swapping $M$ and $N$, so we are done with this case.

If $t\in S$, then we can replace $t$ in all the above arguments with $t-\delta$ for some $\delta$ small enough such that $t-\delta+\epsilon\notin S$, and the above arguments follow to show $\mrk_N(I^\epsilon_{-\epsilon})\preceq_{2\epsilon} \mrk_M(I)$.

Lastly, if $t+\epsilon\in S$, then $\im(\phi_x^N(s-\epsilon\leq t+\epsilon)) = \mrk_N([s-\epsilon,t+\epsilon'))$ for all $\epsilon'=\epsilon+\delta$, $\delta>0$ sufficiently small.
Thus, the above arguments give us $\mrk_N(I^{\epsilon'}_{-\epsilon'})\preceq_{2\epsilon'} \mrk_M(I)$ for all such $\epsilon'$, and when taking the infimum in \autoref{def:erosion-distance}, we get $\dE(\mrk_M,\mrk_N)\leq \epsilon$, as desired.
\end{proof}

\section{Detailed Proofs for Meta-Diagrams}
\label{sec:meta-diagram-details}

\begin{proof}[Proof of \autoref{prop:mdgmrepresentatives}]
First, we establish the existence of such a pair $(M^+,M^-)$. 
Suppose $A\in \SPvec$ has a representative $(M^+_1,M^-_1)$, with $\mathcal{J}:=\barc(M^+_1)\cap\barc(M^-_1)\neq \emptyset$. 
Define $M_2^+:=\oplus_{I\in \barc(M^+_1)\setminus \mathcal{J}} \kf^I$, $M^-_2:=\oplus_{I\in \barc(M^-_1)\setminus \mathcal{J}} \kf^I$, and $V:=\oplus_{I\in \mathcal{J}} \kf^I$.
Consider $M^+_1\oplus M^-_2\oplus V$ and $M^+_2\oplus M^-_1\oplus V$. 
By construction, both of these have barcode $\barc(M^+_1)\cup \barc(M^-_1)$, where $\cup$ is the multiset union.
Hence, these two modules are isomorphic. 
As a result, in $\SPvec$, we have $[(M^+_1,M^-_1)] = [(M^+_2,M^-_2)]$, and by construction $(M^+_2,M^-_2)$ is a representative with $\barc(M^+_2)\cap \barc(M^-_2)=\emptyset$.

Now we establish the uniqueness of the pair $(M^+,M^-)$. Suppose that $[(M^+_1,M^-_1)]=[(M^+_2,M^-_2)]$, $\barc(M^+_1)\cap\barc(M^-_1)=\emptyset$, and $\barc(M^+_2)\cap \barc(M^-_2)=\emptyset$. 
It is a simple algebraic fact that for two 1-parameter persistence modules $M$ and $N$, $\barc(M\oplus N) = \barc(M)\cup \barc(N)$, where $\cup $ is the multiset union.
By definition of $[(M^+_1,M^-_1)]=[(M^+_2,M^-_2)]$, there must exist a 1-parameter persistence module $V$ such that $M^+_1\oplus M^-_2\oplus V\cong M^+_2\oplus M^-_1\oplus V$.
This implies that $\barc(M^+_1)\cup \barc(M^-_2) = \barc(M^+_2)\cup \barc(M^-_1)$.
By our assumptions on intersections, this implies that $\barc(M^+_1) = \barc(M^+_2)$ and $\barc(M^-_1)=\barc(M^-_2)$, which means $(M^+_1,M^-_1)=(M^+_2,M^-_2)$.
Therefore, this is the unique representative satisfying our intersection criterion.
\end{proof}

\begin{proof}[Proof of \autoref{prop:mdgmmobius}]
Suppose $s=s_i<t=s_j$.
Then we have:
\begin{align*}
    \sum_{\substack{I\in \Dgm\\ I\supseteq [s,t)}} \mdgm(I) &= \sum_{k=j}^n\sum_{h=1}^i \mdgm([s_h,s_k)) + \sum_{h=1}^i \mdgm([s_h,\infty))\\
    &= \sum_{k=j}^n\sum_{h=1}^i \big(\mrk([s_h,s_k)) - \mrk([s_h,s_{k+1})) \\
    &\hphantom{\sum_{k=j}^n\sum_{h=1}^i \big(\ }   + \mrk([s_{h-1},s_{k+1})) - \mrk([s_{h-1},s_k))\big)\\
    &\hphantom{= \ } + \sum_{h=1}^i \big(\mrk([s_h,\infty)) - \mrk([s_{h-1},\infty))\big)\\
    &= \sum_{k=j}^n \big( \mrk([s_i,s_k)) - \mrk([s_i,s_{k+1}))\big) + \mrk([s_i,\infty))\\
    &= \mrk([s_i,s_j))
\end{align*}
Now suppose $s=s_i<t=\infty$.
We have:
\begin{align*}
    \sum_{\substack{I\in \Dgm\\ I\supseteq [s,t)}} \mdgm([s_i,\infty)) &= \sum_{h=1}^i \mdgm([s_h,\infty))\\
    &= \sum_{h=1}^i\big(\mrk([s_h,\infty))-\mrk([s_{h-1},\infty))\big)\\
    &= \mrk([s_i,\infty))
\end{align*}

\end{proof}

\begin{proof}[Proof of \autoref{prop:rectanglemdgm}]
First, note that $M$ is constructible, over some set $S=\{s_1<\ldots<s_4\}$ of size no more (but potentially less than) three, with $S$ consisting of $s,s',t$, and $t’$. It is straightforward to compute the following:
\[\mrk([a,b)) = 
\begin{cases}
[t,t') & \mbox{if }\, s\leq a\leq b< s’;\\ 
0 & \mbox{otherwise},  
\end{cases}\]
as $\im(a\leq b)$ is either the image of $[t,t')$ under the identity, or trivial.

Assume without loss of generality that $s=s_a$ and $s'=s_b$.
If $a,b\notin S\times S$, then immediately $\mdgm([a,b))=0$ by definition.
To compute the remainder of the meta-diagram, for each pair $s_i<s_j$, we need to compute the four meta-ranks $\mrk([s_i,s_j))$, $\mrk([s_i,s_{j+1}))$, $\mrk([s_{i-1},s_{j+1}))$, and $\mrk([s_{i-1},s_j))$.
We now break into cases based on where $s_i,s_j$ are, the domains and codomains of the image maps $\phi$ in the meta-rank definition:
\begin{itemize}
    \item Case 1: $s_i< s$. All four meta-ranks are trivial since the domains $M_{s_i}$, $M_{s_{i-1}}$ are trivial modules. Hence, $\mdgm([s_i,s_j)) = 0$.
    \item Case 2: $s_j>s'$. All four meta-ranks are trivial since the codomains $M_{s_j}$, $M_{s_{j+1}}$ are trivial modules, and $\mdgm([s_i,s_j)) = 0$.
    \item Case 3: $s_i=s, \, s<s_j<s' $. We have
    $\mrk([s_i,s_j)) = [t,t')$, $\mrk([s_i,s_{j+1})) = [t,t')$, $\mrk([s_{i-1},s_{j+1}) = 0$ and $\mrk([s_{i-1},s_j)) = 0$, so $\mdgm([s_i,s_j)) = 0$.
    \item Case 4: $s<s_i < s', s_j = s'$. We have $\mrk([s_i,s_j)) = [t,t')$, $\mrk([s_i,s_{j+1})) = 0$, $\mrk([s_{i-1},s_{j+1})) = 0$ and $\mrk([s_{i-1},s_j)) = [t,t')$, so $\mdgm([a,b]) = 0$.
    \item Case 5: $s<s_i<s_j<s'$. All four meta-ranks are $[t,t')$, so $\mdgm([s_i,s_j)) = 0$.
    \item Case 6: $s_i=s, \, s_j=s'$. We have
    $\mrk([s_i,s_j)) = [t,t')$, \newline $\mrk([s_i,s_{j+1})) = 0$, $\mrk([s_{i-1},s_{j+1})) = 0$ and $\mrk([s_i,s_{j+1})) = 0$, so $\mdgm([s_i,s_j)) = [t,t')$.
\end{itemize}
This exhausts the cases for positions of $s_i$ and $s_j$ relative to $s$ and $s'$, and so we are done.
\end{proof}

\autoref{prop:mdgmdEpseudometric} follows analogously to \autoref{prop:mrk-pseudometric}:
\begin{proposition}
\label{prop:mdgmdEpseudometric}
$\dES$ as defined in \autoref{def:dEmdgm} is an extended pseudometric on the collection of meta-diagrams of $S$-constructible modules $M:\R^2\to \cvec$. 
\end{proposition}

\begin{proof}
[Proof of \autoref{thm:mdgmdEstability}]
To show this, we show $\dES(\mdgm_M,\mdgm_N)\leq \dE(\mrk_M,\mrk_N)$ and then invoke \autoref{thm:mrk-stability}.
Let $\epsilon\geq 0$ and suppose that for all $[s,t)\in \Dgm$, we have $\mrk_M([s-\epsilon,t+\epsilon))^\epsilon\preceq_{2\epsilon} \mrk_N([s,t))$, and $\mrk_N([s-\epsilon,t+\epsilon))^\epsilon\preceq_{2\epsilon} \mrk_M([s,t))$.
Fix $[s_i,s_j)\in \Dgm$.
By our assumption, we have the following four injective maps:
\begin{align*}
    \chi_1:\barc(\mrk_M([s_i-\epsilon,s_j+\epsilon)))^\epsilon[2\epsilon:]&\to \barc(\mrk_N([s_i,s_j)))\\
    \chi_2:\barc(\mrk_M([s_{i-1}-\epsilon,s_{j+1}+\epsilon)))^\epsilon[2\epsilon:] &\to \barc(\mrk_N([s_{i-1},s_{j+1})))\\
    \chi_3:\barc(\mrk_N([s_{i-1}-\epsilon,s_j+\epsilon)))^\epsilon[2\epsilon:]&\to \barc(\mrk_M([s_{i-1},s_j)))\\
    \chi_4:\barc(\mrk_N([s_i-\epsilon,s_{j+1}+\epsilon)))^\epsilon[2\epsilon:]&\to \barc(\mrk_M([s_i,s_{j+1})))
\end{align*}
Let $s_{a}:=\overline{S}_\leq(s_i-\epsilon)$ and $s_b:=\overline{S}_\geq(s_j+\epsilon)$.
Suppose $c:=s_{i+1}-s_i$ (which by assumption is constant for any $1\leq i\leq n-1$).
We then have $s_{a-1} = s_a-c\leq s_i-\epsilon - c = s_{i-1}-\epsilon$.
Similarly, we have $s_{b+1}\geq s_{j+1}+\epsilon$.
This implies, for example, that $\mrk_M([s_{a-1},s_{b+1}))\preceq \mrk_M([s_{i-1}-\epsilon,s_{j+1}+\epsilon))$, and a similar statement holds for the domains of the other three maps $\chi_i'$ above.
Thus, by composing each maps $\chi_i$ above with the map guaranteed by the definition of $\preceq$, we can define:
\begin{align*}
    \chi_1':\barc(\mrk_M([s_a,s_b)))^\epsilon[2\epsilon:]&\to \barc(\mrk_N([s_i,s_j)))\\
    \chi_2':\barc(\mrk_M([s_{a-1},s_{b+1})))^\epsilon[2\epsilon:] &\to \barc(\mrk_N([s_{i-1},s_{j+1})))\\
    \chi_3':\barc(\mrk_N([s_{a-1},s_b)))^\epsilon[2\epsilon:]&\to \barc(\mrk_M([s_{i-1},s_j)))\\
    \chi_4':\barc(\mrk_N([s_a,s_{b+1})))^\epsilon[2\epsilon:]&\to \barc(\mrk_M([s_i,s_{j+1})))
\end{align*}

The multiset union of the four barcodes in the domains form the barcode of \newline $\PN(M,N)([\overline{S}_\leq(s_i-\epsilon),\overline{S}_\geq(s_j+\epsilon))^\epsilon[2\epsilon:] = \PN(M,N)([s_a,s_b)))^\epsilon[2\epsilon:]$, and the multiset union of the four barcodes in the codomains form the barcode of $\PN(N,M)([s_i,s_j))$.
Hence, we can let $\chi:\barc(\PN(M,N))([s_a,s_b))^\epsilon[2\epsilon:]\to \barc(\PN(N,M))([s_i,s_j))$ be the disjoint union of $\chi_i'$ for $1\leq i\leq 4$.
As each $\chi_i'$ is injective and has $\chi_i'(J)\subseteq J$, these properties will hold for $\chi$ as well.
\end{proof}

\begin{remark}
We can remove the condition $S$ is evenly-spaced, but there is a price to pay for doing so.
If $S$ is not evenly-spaced, let $\mathrm{irreg}(S) := (\max_{1\leq i\leq n-1} s_{i+1}-s_i) - (\min_{1\leq i\leq n-1} s_{i+1}-s_i)$.
We can define erosion distance as before, removing only the evenly-spaced condition.
In this setting, the stability result appears as:
\begin{theorem}\label{thm:mdgm general stability}
Suppose $M,N:\R^2\to \cvec$ are $S$-constructible. Then we have:
\[\dES(\mdgm_M,\mdgm_N)\leq \dI(M,N)+ \mathrm{irreg}(S)\]
\end{theorem}
Note that $\mathrm{irreg}(S)=0$ if and only if $S$ is evenly-spaced, so this result generalizes Theorem \ref{thm:mdgmdEstability}.
The main issue when $S$ is not evenly-spaced is that we could have $s_{a-1}>s_{i-1}-\epsilon$, which causes the proof of \autoref{thm:mdgmdEstability} to fail.
However, the additive term $\mathrm{irreg}(S)$ accounts for this.
In particular, set $s_a:=\overline{S}_\leq(s_i-\epsilon-\mathrm{irreg}(S))$, $c_a:=s_a-s_{a-1}$ and $c_i:=s_i-s_{i-1}$. 
By definition, $c_i-c_a\leq \mathrm{irreg}(S)$, so we have:
\[s_{a-1}=s_a-c_a\leq s_i-\epsilon-\mathrm{irreg}(S)-c_a\leq s_i-\epsilon-c_i=s_{i-1}-\epsilon\]
Similarly, setting $s_b:=\overline{S}_\geq(s_j+\epsilon+\mathrm{irreg}(S))$ we get $s_{b+1}\geq s_{j+1}+\epsilon$.
The proof of \autoref{thm:mdgm general stability} then follows similarly to that of \autoref{thm:mdgmdEstability}, upon using our new definitions for $s_a$ and $s_b$.
\end{remark}

\section{Detailed Proofs for Algorithms}
\label{sec:algorithm-details}

\begin{proof}[Proof of \autoref{prop:mrkalg1}]
By definition, $\mrk_M([i,i])$ is the one-dimensional persistence module $M_x^i$ along the vertical slice in $M$ from $(i,1)$ to $(i,n)$.
This is a sub-path of $\gamma_i$, but as the $\gamma_i$ has an initial portion from $(1,1)\to (i-1)$, we need to both the birth and death time of an interval in $F_{\gamma_i}$ to make the indexing align.
After this, taking the intersection of each interval with $[1,n]$ gives the persistence of said interval specifically within the slice $M_x^i$.
\end{proof}

\begin{proof}[Proof of \autoref{prop:mrkalg2}]
From our ordering, we can start with an interval $I^i_i$ in $\mrk_M([i,i])$ and find its corresponding interval $I_k^{i-1}$ in $\mrk_M([k,i-1])$.
By definition, $I_k^{i-1}$ stems from the image of some interval module summand $I_k^i$ from $\mrk_M([k,k]) = M^k_x$.
By commutativity conditions of persistence modules, we have:
\[\phi_x(k\leq i)= \phi_x(i-1\leq i)\circ \phi_x(k\leq i-1).\]
So it suffices to show $\phi_x(i-1\leq i)(I_k^{i-1}) = I_k^{i-1}\cap I_i^i$.
From our ordering and the vineyards algorithm, we know that $\phi_x(i-1\leq i)(I_k^{i-1})\subseteq I_i^i$, and its immediate that this image is also a subset of $I_k^{i-1}$, thus we have $\phi_x(i-1\leq i)(I_k^{i-1})\subseteq I_k^{i-1}\cap I_i^i$.
The only way this image could not be full is if at some point in the update process as we swept from $F_{\gamma_{i-1}}$ to $F_{\gamma_i}$, the interval $I$ in \texttt{intervals} corresponding to both $I_k^{i-1}$ and $I_i^i$ shortened to become a strict subset of $I_k^{i-1}$ and then re-expanded to include all of $I_k^{i-1}$.
If this were to happen, then we would have $\phi_x(i-1\leq i)(I^{i-1}_k)\subsetneq I^{i-1}_k\cap I_i^i$.

We claim this cannot happen.
To see why, note from Figure \autoref{fig:same-pair-updates} that an interval can only shrink by one in its death coordinate in the third or fourth cases.
In either case, $\sigma_{i+1}$ is the death coordinate which moved to arrive earlier in the filtration after the transposition.
By the order in which we are sweeping down through the column, $\sigma_{i+1}$ can not be moved to arrive later in the filtration by any transposition, hence the death coordinate for the interval corresponding to $\sigma_{i+1}$ can not increase again.
Similarly, an interval can only shrink by one in its birth coordinate in the first and third cases.
In either case, $\sigma_i$ is the birth coordinate which moved to arrive later in the filtration after the transposition.
Since $\sigma_i$ is now after $\sigma_{i+1}$ in the new filtration, by the order we are sweeping down the column, $\sigma_i$ will not be touched by another transposition again in the column.
Hence, the birth coordinate for the interval corresponding to $\sigma_i$ cannot decrease again.

From this, we have shown the image $\phi_x(k\leq i)(I_k^k)= I_k^{i-1}\cap I_i^i$, and so we are done.
\end{proof}

\end{document}